\documentclass[11pt]{amsart}
\usepackage[margin=1in]{geometry}
\usepackage{amsmath,amssymb,amsthm}
\usepackage{xcolor}
\usepackage{longtable}
\usepackage[hidelinks]{hyperref}

\DeclareFontFamily{U}{wncy}{}
\DeclareFontShape{U}{wncy}{m}{n}{<->wncyr10}{}
\DeclareSymbolFont{mcy}{U}{wncy}{m}{n}
\DeclareMathSymbol{\Sha}{\mathord}{mcy}{"58}

\newcommand{\Q}{\mathbb{Q}}
\newcommand{\Z}{\mathbb{Z}}
\newcommand{\R}{\mathbb{R}}
\newcommand{\Rk}{\mathfrak{R}}
\newcommand{\tors}{\mathrm{tors}}
\DeclareMathOperator{\rank}{rank}
\DeclareMathOperator{\sqf}{sqf}
\DeclareMathOperator{\ord}{ord}
\DeclareMathOperator{\sign}{sign}
\DeclareMathOperator{\corank}{corank}
\DeclareMathOperator{\Sel}{Sel}
\DeclareMathOperator{\Gal}{Gal}

\theoremstyle{plain}
\newtheorem{theorem}{Theorem}[section]
\newtheorem{proposition}[theorem]{Proposition}
\newtheorem{lemma}[theorem]{Lemma}
\newtheorem{corollary}[theorem]{Corollary}
\newtheorem{conjecture}[theorem]{Conjecture}
\theoremstyle{definition}
\newtheorem{example}[theorem]{Example}
\newtheorem{remark}[theorem]{Remark}

\title{Elliptic curve rank growth in multi-quadratic extensions}
\author{Dylan T. Costa}
\address{University of Connecticut, Department of Mathematics, Storrs, CT 06269, USA}
\email{dylan.costa@uconn.edu}
\urladdr{https://dylan-costa.github.io/}
\author{\'Alvaro Lozano-Robledo}
\address{University of Connecticut, Department of Mathematics, Storrs, CT 06269, USA}
\email{alvaro.lozano-robledo@uconn.edu}
\urladdr{http://alozano.clas.uconn.edu}

\begin{document}

\begin{abstract}
In this article we give an algorithm to produce infinitely many explicit quadratic twists of a given elliptic curve over $\Q$ with a rational point and fixed root number. Moreover, under the Parity Conjecture (or other technical hypotheses), for each $n\geq 1$, we produce explicit multi-quadratic extensions $L_n$ of $\Q$, of degree $2^n$, such that the growth in rank from $\Q$ to $L_n$ is at least $\max\{2n+2^{n-1}-1,\,2^n-1\}$ when the root number of $E$ is $-1$, and at least $\max\{2n,\,2^{n-1}+2n-2\}$ when it is $+1$. As a result, assuming Parity, we exhibit elliptic curves over $\Q$ and extensions of degree $2^n$ where the rank after base-extension is greater than or equal to $32+\max\{2n,\,2^{n-1}+2n-2\}$ or $30+\max\{2n+2^{n-1},\,2^n\}$.
\end{abstract}

\maketitle

\section{Introduction}

Let $K$ be a number field and let $E$ be an elliptic curve defined over $K$. The set of $K$-rational points on $E$ forms a finitely generated abelian group by the Mordell--Weil theorem. Thus, the structure theorem of finitely generated abelian groups provides a group isomorphism
\[E(K)\cong E(K)_\tors\times \Z^{R_{E/K}}\]
where $E(K)_\tors$ is the (finite) subgroup formed by all the elements of finite order and the non-negative integer $R_{E/K}=\rank E(K)\geq 0$ is the rank of $E(K)/E(K)_\tors$ as a $\Z$-module. By Merel's proof of the uniform boundedness conjecture (see \cite{merel}), if we fix $d\geq 1$, and we let $E/K$ be any elliptic curve over a number field $K$ of degree $[K:\Q]=d$, then $E(K)_\tors$ is, up to isomorphism, one of finitely many possibilities. The concrete lists of torsion subgroups that can occur for $d=[K:\Q]=1$, 2, 3, and 4 are known by work of Mazur (for $d=1$; see \cite{mazur}), Kamienny, Kenku, and Momose (for $d=2$; see \cite{kamienny,kenkumomose}), Derickx, Etropolski, van Hoeij, Morrow, and Zureick-Brown (for $d=3$; see \cite{DEvHMZB}, and \cite{jeonschweizer} for rational elliptic curves), and Derickx and Najman (for $d=4$; see \cite{derickxnajman}). In addition, Najman has claimed the classification of torsion subgroups for $d=5$ in a recent preprint (see \cite{najmanquintic}). The complete classification of possible torsion subgroups in degree $d>5$ is not known, though much work has been done in that direction (see \cite{GJNajman}, for instance). One can also ask what are the possible torsion subgroups for a specific number field $K$, and this is known in several cases, such as $K=\Q$ (due to Mazur), or $\Q(i)$, $\Q(\sqrt{-3})$ (due to Najman, \cite{najmancyclotomic}). For a classification result of torsion subgroups in the function field case, see McDonald's \cite{mcdonald}.

In contrast to the progress in the classification of torsion subgroups of elliptic curves, the classification of ranks over number fields continues to be wide open, even for $K=\Q$ (more is known over function fields, see for instance \cite{ulmer}, \cite{griffon}). For a fixed $d\geq 1$ and a number field $K$, we define two sets:
\begin{align*}
\Rk_K&=\{R_{E/K}: E \text{ is an elliptic curve defined over } K\}, \text{ and}\\
\Rk_d&=\{R_{E/L}: E \text{ is an elliptic curve defined over some } L \text{ with } [L:\Q]=d\}.
\end{align*}
Most basic questions about $\Rk_K$ and $\Rk_d$ are open. Are these sets bounded? If $\Rk_K$ is bounded for some field $K$, what is $\max(\Rk_K)$? Similarly, if $\Rk_d$ is bounded for some $d\geq 1$, what is $\max(\Rk_d)$? If, on the contrary, $\Rk_K$ is unbounded for some number field $K$ or $\Rk_d$ is unbounded for some $d\geq 1$, is $\Rk_K=\Z^{\geq 0}$, or is $\Rk_d=\Z^{\geq 0}$? If not, what natural numbers do not occur as the rank of an elliptic curve $E/F$? In the case of $K=\Q$, there are elliptic curves $E/\Q$ with $\rank E(\Q)\geq r$ for every $r\leq 32$ (see Section~\ref{sec-examples}), but the exact ranks of the record curves are only known conditionally: the curves of rank exactly $28$, $30$, $31$ and $32$ in Section~\ref{sec-examples} rely on the Generalized Riemann Hypothesis (GRH) and/or the Birch and Swinnerton-Dyer conjecture (BSD), and for the curves of Theorems~\ref{thm-E31} and \ref{thm-E29} only lower bounds are known. For the latest discoveries of ranks of elliptic curves, see Dujella's website \cite{dujella} and the ICARM database \cite{icarm}.

For a fixed abelian group $T=\Z/a\Z\times\Z/ab\Z$, for some $a,b\geq 1$, we can also ask similar questions about
\begin{align*}
\Rk^1_{K,T}&=\{R_{E/K}: E/K \text{ and } E(K)_\tors\cong T\}, \text{ and}\\
\Rk^1_{d,T}&=\{R_{E/L}: E/L \text{ with } [L:\Q]=d, \text{ and } E(L)_\tors\cong T\}.
\end{align*}
Similarly, we can define $\Rk^\infty_{d,T}$ to denote the sets of ranks $R\geq 1$ such that there are infinitely many non-isomorphic (over $\overline{\Q}$) elliptic curves $E/L$, with torsion isomorphic to $T$, and rank $R$, where $L$ runs over all fields with $[L:\Q]=d$.

In this paper we show the following results, where $\sup(H)$ denotes the supremum of the elements of $H$.

\begin{theorem}\label{thm-main1}
Let $T$ be one of the $15$ possible isomorphism types for a torsion subgroup of an elliptic curve over $\Q$, let $\ast\in\{1,\infty\}$, let $\Rk^\ast_{d,T}$ be the set defined above, and assume the Parity Conjecture for elliptic curves over $\Q$. Then, for any $n\geq 1$, we have
\begin{enumerate}
\item if $\sup(\Rk^\ast_{1,T})$ is odd, then $\sup(\Rk^\ast_{2^n,T})\geq \sup(\Rk^\ast_{1,T})+\max\{2n+2^{n-1}-1,2^n-1\}$, and
\item if $\sup(\Rk^\ast_{1,T})$ is even, then $\sup(\Rk^\ast_{2^n,T})\geq \sup(\Rk^\ast_{1,T})+\max\{2n,2^{n-1}+2n-2\}$.
\end{enumerate}
Moreover, the implicit extensions of degree $2^n$, respectively, can be explicitly constructed for each curve achieving the supremum over $\Q$.
\end{theorem}

Our results are based on the following construction, and the key here is that the sequence $\{d_k\}_{k\geq 1}$ can be constructed explicitly, as we will discuss in examples.

\begin{theorem}\label{thm-main2}
Let $E/\Q$ be an elliptic curve with conductor $N_E$, global root number $W(E/\Q)$, and rank $R_{E/\Q}$. Fix a sequence of signs $\varepsilon_k\in\{\pm1\}$, $k\geq1$. Then, there exists an explicitly constructible sequence of square-free integers $\{d_k\}_{k\geq1}$ such that $L_n=\Q(\sqrt{d_1},\ldots,\sqrt{d_n})$ has degree $2^n$ over $\Q$, $R_{E^{d_k}/\Q}\geq 1$ and $W(E^{d_k})=\varepsilon_k$ for every $k=1,\ldots,n$. Moreover,
\begin{enumerate}
\item $R_{E/L_n}\geq R_{E/\Q}+n$.
\item If $W(E)=-1$ and $W(E^{d_k})=1$ for every $k$, and the Parity Conjecture holds for $E^d$ for every $d\in\mathcal{D}_n=\langle d_1,\ldots,d_n\rangle\subseteq\Q^\times/(\Q^\times)^2$ with $d\neq 1$, then
\[R_{E/L_n}\geq R_{E/\Q}+2n+2^{n-1}-1.\]
\item If $W(E)=-1$ and $W(E^{d_k})=-1$ for every $k$, then $W(E^d)=-1$ for every $d\in\mathcal{D}_n$, and either
\begin{enumerate}
\item $\ord_{s=1}L(E^d,s)=1$ for every $d\in\mathcal{D}_n$ with $d\neq1$, in which case
\[R_{E/L_n}=R_{E/\Q}+2^n-1,\]
or
\item $\ord_{s=1}L(E^d,s)>1$ for some $d\in\mathcal{D}_n$ with $d\neq 1$. In this case, if the Parity Conjecture holds for $E^d$ for every $d\in\mathcal{D}_n$ with $d\neq1$, or if for every such $d$ there is a prime $p$ such that $\Sha(E^d/\Q)[p^\infty]$ is finite, then
\[R_{E/L_n}\geq R_{E/\Q}+2^n-1.\]
If BSD holds, then in fact $R_{E/L_n}\geq R_{E/\Q}+2^n+1$.
\end{enumerate}
\item Suppose $W(E)=1$ and assume the Parity Conjecture holds for $E^d$ for every $d\in\mathcal{D}_n$ with $d\neq 1$, or that for every such $d$ there is a prime $p$ such that $\Sha(E^d/\Q)[p^\infty]$ is finite. Then
\begin{enumerate}
\item if $W(E^{d_k})=1$ for every $k$, we have $R_{E/L_n}\geq R_{E/\Q}+2n$,
\item if $W(E^{d_k})=-1$ for every $k$, we have $R_{E/L_n}\geq R_{E/\Q}+2^{n-1}$, and
\item if $n\geq2$, $W(E^{d_k})=1$ for $1\leq k\leq n-1$ and $W(E^{d_n})=-1$, we have $R_{E/L_n}\geq R_{E/\Q}+2^{n-1}+2n-2$.
\end{enumerate}
\end{enumerate}
\end{theorem}

When we apply our results to the curves $E_{32}$ and $E_{31}'$ of Section~\ref{sec-examples}, we obtain the following corollary.
\begin{corollary}\label{cor-highrank}
Let $n\geq 1$. There exist explicit multi-quadratic number fields $L_n$ and $L_n'$ of degree $2^n$ such that:
\begin{enumerate}
\item $R_{E_{32}/L_n}\geq 32+n$, and, if the Parity Conjecture holds for the twists $E_{32}^d$ with $d\in\mathcal{D}_n$, $d\neq 1$, then $R_{E_{32}/L_n}\geq 32+\max\{2n,2^{n-1}+2n-2\}$;
\item $R_{E_{31}'/L_n'}\geq 31+n$, and, if the Parity Conjecture holds for the twists $E_{31}'^{\,d}$ with $d\in\mathcal{D}_n$, $d\neq 1$, then $R_{E_{31}'/L_n'}\geq 30+\max\{2n+2^{n-1},2^n\}$.
\end{enumerate}
In both cases, the torsion subgroup does not grow from $\Q$ to $L_n$ (resp.\ $L_n'$).
\end{corollary}

The paper is organized as follows. In Section~\ref{sec-examples} we give a number of examples of elliptic curves with the highest ranks known as of the writing of this paper, and provide some examples that illustrate the algorithm introduced in this article. In Section~\ref{sec-previous} we give an overview of previous results of rank growth in multi-quadratic extensions. In Section~\ref{sec-rootnumbers} we cite previous results in the literature on root numbers and ranks. In Section~\ref{sec-proofs} we first prove some key lemmas, and then proceed to put together the proofs of the main theorems. In Section~\ref{sec-code} we explain the code that the reader can find in our GitHub repository for this project, and we give more detailed examples of how the algorithm works. Finally, in Section \ref{sec-rank3}, we discuss work in progress to extend the constructions of this paper to quadratic twists of rank $\geq 3$ for some families of elliptic curves.

\subsection{Acknowledgements} We would like to thank Enrique Gonz\'alez-Jim\'enez for several conversations about growth of ranks in extensions, and Keith Conrad for helpful comments and suggestions. Our special thanks to Andrej Dujella for giving us very useful comments and suggestions on an earlier version of this paper. The authors would also like to thank Noam Elkies, Ritik Jain, Zev Klagsbrun, and Samuel Li for making the curves they have discovered publicly available. We would also like to thank ICARM (which is supported by NSF Grant DMS 2425401). Code for this article can be found at \href{https://github.com/dylan-costa/rank-growth-multiquadratic}{github.com/dylan-costa/rank-growth-multiquadratic}.

\subsection{Declaration of AI use} We have used Claude Opus 5.5 for literature review, to help us write Magma code, and to find examples of square-free values and quadratic twists with certain properties. We used Claude to provide a referee report on our first draft of this paper and made minor changes based on its recommendations. The main ideas of the paper were conceived before the authors even started using LLMs as a tool.

\section{Examples of Elliptic Curves of High Rank and Rank Growth}\label{sec-examples}

The starting point for our investigation is the following examples of high rank elliptic curves. Elliptic curves over $\Q$ of rank $\geq 31$ were first found by Levent Alp\"oge and Ava Howell, with the help of Claude, and announced at the ICARM database \cite{icarm}. Later, Ritik Jain and Samuel Li found several additional elliptic curves of rank $\geq 31$ with smaller conductors, and Elkies, Klagsbrun and Li found a curve of rank $\geq 32$. The curves of Theorems~\ref{thm-E32}, \ref{thm-E31} and \ref{thm-E31prime} are \href{https://elliptic-rank.icarm.cloud/curve/4239}{\#4239}, \href{https://elliptic-rank.icarm.cloud/curve/3900}{\#3900} and \href{https://elliptic-rank.icarm.cloud/curve/302}{\#302} at the ICARM website, respectively.

\begin{theorem}[Elkies, Klagsbrun, Li, 2026]\label{thm-E32}
Assuming BSD and GRH, the elliptic curve $E_{32}/\Q$ given by $y^2+xy=x^3+Ax+B$ with
\begin{align*}
A=&-28623928991398919165500467001947502452240040593445344054331\\
&\ 428726899413\\
B=&\ 16366380970220279630250861194633604823688757308892533348756934401811\\
&\ 43871885452358966522018914326964558817
\end{align*}
has $E_{32}(\Q) \cong \Z^{32}$.
\end{theorem}
About the curve $E_{32}$: Elkies, Klagsbrun, and Li discovered the elliptic curve $E_{32}$ and showed $R_{E_{32}/\Q} \geq 32$, and Hackfort has claimed that BSD and GRH imply the rank is exactly $32$.

\begin{theorem}[Jain, 2026]\label{thm-E31}
The elliptic curve $E_{31}/\Q$ given by $y^2+xy=x^3+Ax+B$ with
\begin{align*}
A=&-39404268000900982457292816039978170606307632627450684342785\\
B=&\ 29972230203581841024182729278368050094604005692542359874884023483994\\
&\ 80349577137103225097
\end{align*}
satisfies $\rank E_{31}(\Q)\geq 31$.
\end{theorem}

In the announcement for the following curve, Alp\"oge and Howell claim that, under BSD and GRH, they can show that the rank is exactly 31.

\begin{theorem}[Alp\"oge--Howell, 2026]\label{thm-E31prime}
Assuming BSD and GRH, the elliptic curve $E_{31}'/\Q$ given by $y^2+xy+y=x^3+x^2+Ax+B$ with
\begin{align*}
A=&-1284727764113567728281797636015784768866707681415849262157224232063\\
B=&\ 56036832145426133925685933890191531233276985868494540685804386919945\\
&\ 6710681989058863306170127006181
\end{align*}
has $E_{31}'(\Q)\cong\Z^{31}$.
\end{theorem}

\begin{theorem}[Elkies, Klagsbrun, 2025]\label{thm-E30}
Assuming BSD and GRH, the elliptic curve $E_{30}/\Q$ given by $y^2+xy=x^3+Ax+B$ with
\begin{align*}
A=&-12892599774455576272301592959047823530919513428112484011550\\
B=&\ 56075504634839541297708882499950389061755885668763698122393566284838\\
&\ 6484980312656296132
\end{align*}
has $E_{30}(\Q)\cong\Z^{30}$.
\end{theorem}

\begin{theorem}[Li, 2026]\label{thm-E29}
The curve $E_{29}/\Q$ given by $y^2+xy=x^3+Ax+B$ with
\begin{align*}
A=&-2695272471630565276270102221032380764896680795950115\\
B=&\ 580706454307468239309390395193457967255167244398649255950077\\
&\ 98344169481404417
\end{align*}
satisfies $\rank E_{29}(\Q)\geq 29$.
\end{theorem}

\begin{theorem}[Elkies, 2006; \cite{elkies3lectures},\cite{KSW}]\label{thm-E28}
Assuming GRH, the elliptic curve $E_{28}/\Q$ given by
\begin{align*}
&y^2+xy+y=x^3-x^2-20067762415575526585033208209338542750930230312178956502x\\
&\quad+3448161179503055646703298569039072037485594435931918036126600829629\\
&\qquad 1939448732243429
\end{align*}
has $E_{28}(\Q)\cong\Z^{28}$.
\end{theorem}

The following curve is currently one of the highest rank elliptic curves known over a quadratic field.

\begin{theorem}[Elkies, 2025; \cite{elkiesrank34}]\label{thm-E17}
Assuming GRH, the elliptic curve $E_{17}/\Q$ given by
\[y^2=x^3-908800736629952526116772283648363\]
has $E_{17}(\Q)\cong\Z^{17}$ and $E_{17}(\Q(\sqrt{-3}))\cong\Z^{34}$.
\end{theorem}

\begin{remark}
    Let $E_{32}/\Q$ be the curve of Theorem \ref{thm-E32}. In a recent preprint (see \cite{CostaLozanoRobledo-rank35}), the authors have found a quadratic twist $E_{32}^D/\Q$ of $E_{32}$ of rank $\geq 2$ and root number $-1$, where $D$ is an integer with $411$ digits. Thus, according to the Parity Conjecture and together with Prop. \ref{prop-ranksum}, we would have $E_{32}/\Q(\sqrt{D})$ has rank $\geq 35$.
\end{remark}

This article is concerned with rank increase in (multi-quadratic) extensions. The following examples show some particular cases where we get rank increases on an elliptic curve when we base extend to a quadratic field.

\begin{example}\label{ex-rank0twist}
Let $E/\Q$ be an elliptic curve of rank $R$ over $\Q$, and let $d$ be a non-zero square-free integer such that the quadratic twist $E^d/\Q$ of $E$ is a curve of rank 0 (for a definition of quadratic twist see the beginning of Section~\ref{sec-twists}). Note that by the quadratic-twist non-vanishing theorem of Friedberg--Hoffstein and the rank-zero theorem of Gross-Zagier-Kolyvagin--Logachev (see Theorems~\ref{thm-bsdrank01} and \ref{thm-twistsrank0} below), every elliptic curve $E/\Q$ has infinitely many quadratic twists $E^d$ with $\rank E^d(\Q)=0$. Therefore, if we let $A=E^d$ with $E$ as before, then note that $A^d\cong(E^d)^d\cong E$ over $\Q$, and therefore
\begin{enumerate}
\item $A/\Q$ is an elliptic curve of rank zero,
\item its quadratic twist $A^d/\Q$ is an elliptic curve of rank $R$, and
\item $A/\Q(\sqrt{d})$ is an elliptic curve of rank $R$ over $\Q(\sqrt{d})$, by Prop.~\ref{prop-kida}.
\end{enumerate}
For example, if $E_{31}/\Q$ is the curve defined in Theorem~\ref{thm-E31}, then there exists a quadratic twist $A=E_{31}^d$ of rank 0 (finding a single value $d$ that works, though, is quite challenging computationally) such that the rank of $A$ over $\Q$ jumps to 31 when we base extend from $\Q$ to $\Q(\sqrt{d})$.
\end{example}

The next example will illustrate the methods of this paper. More details will be provided in Example~\ref{ex-x3minus2} after we prove our results.

\begin{example}\label{ex-x3minus2-intro}
Consider the elliptic curve $E:y^2=x^3-2$ with $E(\Q)\cong\Z$, generated by the point $(3,5)$, and with LMFDB label \texttt{1728.o3}. The following lists of square-free integers $d$ give three quadratic twists $E^d$ of the given rank:
\begin{itemize}
\item Rank 0 twists: $d=2,5,6$.
\item Rank 1 twists: $d=-1,3,7$.
\end{itemize}
While the rank of $E(\Q(\sqrt{d}))$ for $d=2,5,6$ is one (the rank of $E(\Q)$), one can verify (using Prop.~\ref{prop-kida}) that the rank of $E(\Q(\sqrt2,\sqrt5,\sqrt6))$ is 4 (the rank contributions happen for $d=1,3,10,30$, each contributing $1$). Similarly, the rank of $E(\Q(\sqrt{d}))$ is $2$ for $d=-1,3,7$, but the rank of $E(\Q(\sqrt{-1},\sqrt3,\sqrt7))$ is 8.

In this paper, we aim to construct (explicitly) quadratic twists of rank $\geq 2$. The twists of rank $\geq 2$ of this particular curve with $|d|\leq 62$ are $d=-29,-41,-53,53,62$, each of rank exactly $2$. The values $d=-29,62,123$ are all twists of rank $2$ produced by a naive search over the integers $x$ (the square-free parts of $f(x)$ for $|x|\leq 7$, discarding the twists by $-2$ and $6$, of rank $0$), and one can show that the rank of $E(\Q(\sqrt{-29},\sqrt{62},\sqrt{123}))$ is 12. Since $62$ and $123$ are not relatively prime to $N_E=1728$, our theorem does not apply to this field. However, the root numbers follow the same pattern as in Lemma~\ref{lem:multiquadratic-rank}(2): $W(E^d)=+1$ for $d=-29,62,123$ and $-221154$, and $W(E^d)=-1$ for the three pairwise products. So Parity predicts that the rank is $\geq 10$ (the twist by $-1798$ has rank 3 instead of 1 as one would expect just from parity).

Our theorem produces values
\begin{align*}
    d_1 & =-4407871453975949551199\\
    d_2 & =-21268670741331361021169357399\\
    d_3 & =-513363857667648563737041085212527519
\end{align*}
such that $E(\Q(\sqrt{d_1},\sqrt{d_2},\sqrt{d_3}))$ has rank $\geq 10$ assuming parity. Obviously, the explicit values produced by the theorem are only of theoretical value, but the key is that the algorithm produces an \emph{explicit} construction of values $d_k$, with $k\geq 1$, such that $E(\Q(\sqrt{d_k}:k=1,\ldots,n))$ has rank $\geq 2n+2^{n-1}$ under parity. See Example~\ref{ex-x3minus2} for more details.
\end{example}

\begin{example}\label{ex-x3x2}
The algorithm proved in this paper also works for elliptic curves $E/\Q$ with a given torsion subgroup. Take, for example, $E:y^2=x^3-x^2-9x+9$, which is an elliptic curve with $E(\Q)_\tors\cong\Z/2\Z\times\Z/2\Z$ and rank 1. The rank of $E(\Q(\sqrt{d}))$ for $d=15,-35,-47$ is 3, and the rank of $E(\Q(\sqrt{15},\sqrt{-35},\sqrt{-47}))$ is 12, while the torsion remains $\Z/2\Z\oplus\Z/2\Z$. The torsion does grow over some quadratic fields: over $\Q(\sqrt{-2})$ it becomes $\Z/2\Z\times\Z/4\Z$, since the $2$-torsion point $(1,0)$ is halved there, as $1-3=-2$ and $1+3=4$. Our methods, instead, produce
\begin{align*}
d_1 & =-3651379397538713876111237249085635399\\
d_2 & =-17624510871075773916049759505271880999507199,\\
d_3 & =-425412847147925968212830332740563816177071867083719
\end{align*}

with $E(\Q(\sqrt{d_1},\sqrt{d_2},\sqrt{d_3}))$ of rank $\geq 10$, assuming parity, and full 2-torsion (the torsion is exactly $\Z/2\Z\times\Z/2\Z$ over this field by Corollary~\ref{cor-fujita}, since every non-trivial $d\in\mathcal{D}_3$ is $\equiv 1\bmod 8$, prime to $210N_E$ and not $\pm1$, so it has a prime factor not dividing $210N_E$).
\end{example}

\begin{example}\label{ex-torsiongrowth}
Finally, we want to provide an example of a curve $E/\Q$ where the torsion subgroup grows in a multi-quadratic extension, but it does not grow in any particular quadratic extension. Let $E/\Q$ be given by $y^2=x(x-2)(x+3)=x^3+x^2-6x$ with LMFDB label \texttt{480.e3}. The torsion over $\Q$ is isomorphic to $\Z/2\Z\times\Z/2\Z$ and there is no growth in any quadratic field. However
\[E(\Q(\sqrt2,\sqrt5))\cong\Z/2\Z\times\Z/4\Z\]
and
\[E(\Q(i,\sqrt2,\sqrt3,\sqrt5))\cong\Z/4\Z\times\Z/4\Z.\]
In our results we use Theorem~\ref{thm-fujita} and Cor.~\ref{cor-fujita} to construct multi-quadratic extensions where the torsion subgroup does not grow.
\end{example}

\section{Previous Results on Rank Growth}\label{sec-previous}

A number field is \emph{multiquadratic} if it is a compositum of quadratic fields. Let $G$ be a finite subgroup of $\Q^\times/\Q^{\times2}$, and put $K=\Q(\sqrt{G})=\Q(\{\sqrt{g}:g\in G\})$, so that $\Gal(K/\Q)\cong G^\vee$ is 2-elementary abelian of order $|G|$. Decomposing the representation $E(K)\otimes\Q$ of $\Gal(K/\Q)$ into characters gives the well-known identity
\begin{equation}\label{eq-kida-group}
\rank E(K)=\sum_{d\in G}\rank E^d(\Q),
\end{equation}
where $E^d$ denotes the quadratic twist of $E$ by $d$ (see Prop.~\ref{prop-kida} for the statement). Thus rank growth in multiquadratic extensions is the same problem as finding twists of positive rank indexed by a subgroup (or a coset of a subgroup) of $\Q^\times/\Q^{\times2}$. In this section we summarize known results in this direction.
\begin{itemize}
\item Mestre (see \cite{mestre}) constructed a parametric quadratic twist of rank $\geq 2$ for certain elliptic curves over $\Q$. We will come back to this construction in Section \ref{sec-rank3} to explain some potential uses and limitations of Mestre's construction.
\item Im and Larsen \cite{ImLarsen-parallelopipeds} show that for every $n\geq1$ there exist an elliptic curve $E/\Q$ and a subgroup $G$ of $\Q^\times/\Q^{\times2}$ of order $|G|=2^n$ such that $E^d$ has positive rank for every nontrivial $d\in G$. By \eqref{eq-kida-group}, we have that $\rank E(\Q(\sqrt{G}))\geq 2^n-1$ for a field of degree $2^n$. Earlier, Im and the second author \cite{ImLozano} gave curves $E$ over a number field $K$ together with polynomials $P_1,P_2\in K[t]$ such that the twists of $E$ by $P_1$, $P_2$ and $P_1P_2$ all have positive rank over $K(t)$. Specializing $t$ produces biquadratic extensions $K(\sqrt{P_1(t_0)},\sqrt{P_2(t_0)})$ in which all three nontrivial twists contribute to the rank. As a consequence, for every $\sigma$ in the absolute Galois group of $K$, the curve $E$ has infinite rank over the fixed field of $\sigma$ in the maximal abelian extension of $K$.
\item Park, Poonen, Voight and Wood \cite[\S12.4]{PPVW} prove that for every $n\geq0$ there is a multiquadratic field $K$ of degree $2^n$ such that a positive proportion of all elliptic curves $E/\Q$ satisfy $\rank E(K)\geq 2^n=[K:\Q]$. In particular, $\limsup_E\rank E(K)\geq[K:\Q]$, where $E$ runs over elliptic curves over $K$. They prove their results by combinining \eqref{eq-kida-group} with the density Hales--Jewett theorem of Furstenberg and Katznelson and with the result of Bhargava, Skinner and Zhang that a positive proportion of curves have rank one. Their argument, however, does not identify which of finitely many candidate fields $K$ works, although their compositum can be given explicitly. In the same paper they also obtain, from Heegner points in anticyclotomic towers, fields $K$ of unbounded degree with $\limsup_E\rank E(K)\geq[K:\Q]/2$.
\item T. Dokchitser and V. Dokchitser \cite{DDallquadratic} show, assuming the Birch and Swinnerton-Dyer conjecture, that there are elliptic curves $E$ over number fields $K$ all of whose quadratic twists have positive rank. They characterize such curves by their local behaviour and by the Galois action on the Tate module. For such $E$, the analogue of \eqref{eq-kida-group} over $K$ forces $\rank E(L)\geq\rank E(K)+[L:K]-1$ for every multiquadratic extension $L/K$, still conditionally on BSD. In particular, Goldfeld's conjecture on the average rank of quadratic twists does not extend verbatim to number fields.
\item Bartel \cite{bartel} exhibits a large family of semistable elliptic curves $E/\Q$ whose 2-Selmer group grows in every biquadratic extension of $\Q$. For odd $p$, the $p$-Selmer group grows in every $D_{2p}$-extension and every elementary abelian $p$-extension of rank at least 2.
\item Let $\Q(2)$ denote the compositum of all quadratic extensions of $\Q$. For $E/\Q$, infinite rank over $\Q(2)$ follows from \eqref{eq-kida-group} together with the existence of infinitely many square classes $d$ with $\rank E^d(\Q)\geq1$. These can be taken to be the square-free parts of the values $f(x_0)$, where $E:y^2=f(x)$ and $x_0\in\Q$. Indeed, these values lie in infinitely many distinct square classes (explicitly by Theorem~\ref{maintheorem}, which produces infinitely many independent ones), and the point $(x_0,m)$ on $E^d:dy^2=f(x)$, with $f(x_0)=dm^2\neq 0$, is not a $2$-torsion point, so it has infinite order whenever $E^d(\Q)_\tors=E(\Q)[2]$, that is, for all but at most three square classes $d$ by Theorem~\ref{thm-najmantwists}. Bruin and Najman \cite{BruinNajman} ask whether every elliptic curve $E$ over a number field $K$ has infinite rank over $K\Q(2)$. They show that a positive answer would imply that, for every $n\geq2$, some elliptic curve over a number field has infinitely many quadratic twists of rank a positive multiple of $2^n$. In the same paper they prove that $\rank A(L)-\rank A(K)\neq1$ for an abelian variety $A/K$ whenever $\Gal(L/K)$ has no subgroup of index 2. Over the larger field $\Q^{\rm ab}$, infinite rank is known: Frey and Jarden \cite{FreyJarden1974} proved it for every $E/\Q$, and Im and Larsen \cite{ImLarsen-Qab} extended it to curves over quadratic fields with $j\neq0,1728$.
\end{itemize}

The construction in the present paper also produces subgroups of twists of positive rank, but explicitly. It gives square-free $d_1,\ldots,d_n$, independent modulo squares, each with $W(E^{d_k})=+1$ and an explicit point of infinite order on $E^{d_k}$. Hence $\rank E^{d_k}(\Q)\geq 2$ for each $k$, assuming the parity conjecture. In contrast to the constructions summarized above, the field $\Q(\sqrt{d_1},\ldots,\sqrt{d_n})$ is determined explicitly by our algorithm.

In order to control the growth of the torsion subgroup of $E/\Q$ when base-extended to a quadratic extension $K/\Q$, we will make use of the following theorems of Fujita, Gonz\'alez-Jim\'enez and Tornero, Laska and Lorenz, and Najman, about the growth of torsion in quadratic extensions and torsion subgroups of quadratic twists (see \cite{GJT2} for further refined results in this direction).

\begin{theorem}[\cite{GJT2}, Corollary 4]\label{thm-enrique}
Let $E/\Q$ be an elliptic curve. Then, there are at most $4$ distinct quadratic fields such that $E(\Q)_{tors}\subsetneq E(K)_{tors}$.
\end{theorem}

\begin{theorem}[\cite{najmantwists}, \cite{najmanerratum}, Theorem 1]\label{thm-najmantwists}
Let $E/\Q$ be an elliptic curve. Then, the number of values $d\in\Q^\times/(\Q^\times)^2$ such that $E(\Q)[2]\subsetneq E^d(\Q)_{tors}$ is at most $3$.
\end{theorem}

\begin{theorem}[Laska--Lorenz \cite{LaskaLorenz}, Fujita \cite{fujita1,fujita2}]\label{thm-fujita}
Let $E/\Q$ be an elliptic curve, and let $\Q(2)=\Q(\{\sqrt{d}:d\in\Z\})$. Then, $E(\Q(2))_\tors$ is isomorphic to one of the following $20$ groups:
\[\begin{array}{ll}
\Z/N\Z, & N=1,3,5,7,9,15,\\
\Z/2\Z\times\Z/2N\Z, & N=1,2,3,4,5,6,8,\\
\Z/3\Z\times\Z/3\Z, & \\
\Z/4\Z\times\Z/4N\Z, & N=1,2,3,4,\\
\Z/2N\Z\times\Z/2N\Z, & N=3,4.
\end{array}\]
\end{theorem}

\section{Background on Root Numbers and Quadratic Twists}\label{sec-rootnumbers}

In this section, we collect several results concerning root numbers of elliptic curves over $\Q$. For a complete treatment of root numbers we refer the reader to \cite{rohrlichPCMI}. For an elliptic curve $E/\Q$ the associated global root number, denoted $W(E/\Q)\in\{\pm1\}$, is defined as a product
\[W(E/\Q)=\prod_{p\leq\infty}W_p(E/\Q)\]
where $W_p(E/\Q)$ is the local root number, and the product runs through all the finite and infinite places of $\Q$. The local root numbers $W_p(E/\Q)$ are, in turn, defined in terms of the epsilon factors of the Weil--Deligne representation of $\Q_p$, and their precise definition can be found in \cite{rohrlichPCMI} (see also \cite{rohrlichgalois}). In \cite{rohrlich}, Rohrlich gives explicit formulas for $W_p(E/\Q)$ according to the reduction type of $E$ at $p$. In particular, the formulas show that $W_p(E/\Q)=1$ for all but finitely many places $p$ and the global root number is well defined.

\begin{proposition}[\cite{rohrlich}, Prop. 1, 2, and 3]\label{prop-rohrlich}
Let $E$ be an elliptic curve over $\Q$ and $p\leq\infty$ be a place of $\Q$. Then:
\begin{enumerate}
\item[(a)] $W_\infty(E/\Q)=-1$.
\item[(b)] If $E$ has good reduction at $p$, then $W_p(E/\Q)=1$.
\item[(c)] If $E$ has split multiplicative reduction, then $W_p(E/\Q)=-1$.
\item[(d)] If $E$ has non-split multiplicative reduction, then $W_p(E/\Q)=1$.
\item[(e)] If $p>2$ and $E$ has additive potential multiplicative reduction over $\Q_p$, then $W_p(E/\Q)=\left(\frac{-1}{p}\right)$, the Legendre symbol of $-1$ on $p$.
\item[(f)] If $p>3$ and $E$ has additive potential good reduction over $\Q_p$, and we let $\Delta_p$ be the minimal discriminant of $E/\Q_p$, and $e=12/\gcd(\ord_p(\Delta_p),12)$, then
\[W_p(E/\Q)=\begin{cases}
\left(\frac{-1}{p}\right) & \text{if } e=2 \text{ or } 6,\\[4pt]
\left(\frac{-3}{p}\right) & \text{if } e=3,\\[4pt]
\left(\frac{-2}{p}\right) & \text{if } e=4.
\end{cases}\]
The value $e=1$ does not occur for additive reduction.
\end{enumerate}
\end{proposition}

For the remaining cases for $p=2$ and $p=3$, we refer the reader to \cite{halberstadt} and \cite{rizzo}. With the formulas indicated above and given the conductor of an elliptic curve $E/\Q$, one can compute the global root number in most cases. We recall that the conductor $N_E$ of an elliptic curve is defined as the product
\[N_E=\prod_{p<\infty}p^{f_p}\]
where
\[f_p=\begin{cases}0&\text{if } E \text{ has good reduction at } p,\\ 1&\text{if } E\text{ has multiplicative reduction at }p,\\ 2+\delta_p&\text{if } E \text{ has additive reduction at }p,\end{cases}\]
and $\delta_p=0$ for all $p>3$ and measures wild ramification for $p=2,3$ (see \cite[Appendix C]{silverman}.)

\begin{example}\label{ex-390}
Consider the elliptic curve
\[y^2+xy+y=x^3-289x+3092\]
with $N_E=390=2\cdot3\cdot5\cdot13$, and LMFDB label \texttt{390.c4} (see \cite{lmfdb}). For the infinite place, we always have $W_\infty(E/\Q)=-1$. In order to compute the global root number, it remains to consider the finite primes dividing the conductor. At $p=2,3,5,13$ we have that $\nu_p(N_E)=1$, and thus $E$ has multiplicative reduction at $p$. Moreover, one computes that for $p=3$ the curve has split multiplicative reduction, and at $p=2,5,13$ the curve has non-split multiplicative reduction. Thus,
\begin{align*}
W(E/\Q)&=W_2(E/\Q)W_3(E/\Q)W_5(E/\Q)W_{13}(E/\Q)W_\infty(E/\Q)\\
&=(1)(-1)(1)(1)(-1)\\
&=1.
\end{align*}
Assuming the Parity Conjecture, which we state next (see Conjecture~\ref{conj-parity}), the algebraic rank must satisfy $R_{E/\Q}\equiv0\bmod 2$, and one can indeed verify using descent techniques that $R_{E/\Q}=0$.
\end{example}

\begin{conjecture}[Parity Conjecture]\label{conj-parity}\label{parity}
Let $E/\Q$ be an elliptic curve with rank $R_{E/\Q}$, and let $W(E/\Q)$ be the global root number of $E/\Q$. Then,
\[W(E/\Q)=(-1)^{R_{E/\Q}}.\]
\end{conjecture}

\begin{example}\label{ex-rootnumber-E32}
Consider $E_{32}/\Q$ as defined in Theorem~\ref{thm-E32} which is of rank $R_{E_{32}/\Q}\geq 32$. Its conductor is
\[N_E=2\cdot3\cdot5^2\cdot11^2\cdot13\cdot17\cdot23\cdot31\cdot37\cdot41\cdot43\cdot53\cdot89\cdot p_1\cdot p_2\]
where
\begin{align*}
p_1&=2220713642109042278879558804334283,\\
p_2&=3500766112881427717056865531427442780770809079495985232632875772\\
&\quad 24658029115989670361016516064255191976556239262775289081591134547.
\end{align*}
The curve has multiplicative reduction at every prime $p\mid N_E$ except $p=5$ and $p=11$, and at those primes Proposition~\ref{prop-rohrlich} gives $W_p(E_{32}/\Q)=-1$ if the reduction is split and $W_p(E_{32}/\Q)=1$ if it is non-split. At $p=5$ and $p=11$ the reduction is additive, of Kodaira types $\mathrm{I}_0^*$ and $\mathrm{IV}$ respectively, and potentially good, since $\ord_p(j(E_{32}))\geq0$. For $p=5$ we have $\ord_5(\Delta)=6$, so $e=2$ and $W_5(E_{32}/\Q)=\left(\frac{-1}{5}\right)=1$. For $p=11$ we have $\ord_{11}(\Delta)=4$, so $e=3$ and $W_{11}(E_{32}/\Q)=\left(\frac{-3}{11}\right)=-1$. Using Magma (\cite{magma}), we can readily compute
\[W_p(E_{32}/\Q)=\begin{cases}1&p=5,13,17,23,31,43,53,89,p_1,p_2,\\-1&p=2,3,11,37,41,\infty.\end{cases}\]
Therefore, $W(E_{32}/\Q)=(-1)^6=1$, in agreement with the Parity Conjecture since $R_{E_{32}/\Q}\geq32$ is expected to be equal to 32.
\end{example}

\subsection{Root numbers of quadratic twists.}\label{sec-twists} Our results depend on the behavior of the ranks of quadratic twists of a fixed elliptic curve over $\Q$, so we now state two results concerning their root numbers.

Given an elliptic curve $E/\Q$ with (short) Weierstrass model $y^2=x^3+Ax+B$, we denote by $E^d$ the quadratic twist of $E$ by $d\in\Z$, and we remind the reader that $E^d$ is given by the Weierstrass equation
\[E^d:y^2=x^3+d^2Ax+d^3B.\]
or, alternatively, by the model $dy^2=x^3+Ax+B$. The following result explains how the rank of an elliptic curve defined over $\Q$ grows in multi-quadratic extensions (and generalizes \cite[Ch. X, Exercise 10.16]{silverman}).

\begin{proposition}[\cite{kida}, Prop. 1]\label{prop-kida}\label{prop-ranksum}
Let $E/\Q$ be an elliptic curve and let $d_1,\ldots,d_n$ be integers such that $K_n=\Q(\sqrt{d_1},\ldots,\sqrt{d_n})$ has degree $[K_n:\Q]=2^n$. Let $R_{E/\Q}$ and $R_{E^{d_i}/\Q}$, for each $1\leq i\leq n$, be the ranks of $E(\Q)$ and $E^{d_i}(\Q)$, respectively, and let $R_{E/K_n}$ be the rank of $E(K_n)$. Let
\[\mathcal{D}_n:=\left\{d\in\Z\ \middle|\ d=\prod_{i=1}^nd_i^{e_i},\ e_i\in\{0,1\}\right\}.\]
Then, the rank of $E$ over $K_n$ is given in terms of the ranks of the quadratic twists $E^{d_i}$ over $\Q$ by
\[R_{E/K_n}=\sum_{d\in\mathcal{D}_n}R_{E^d/\Q}.\]
\end{proposition}

In certain cases, the root number of a quadratic twist can be easily computed from the root number of the original curve, as the following two results show.

\begin{lemma}[\cite{dokchitsernonabelian}, Corollary 2]\label{kronlemma}
Suppose $E$ is an elliptic curve over $\Q$, let $N_E$ be the conductor of $E/\Q$ and let $d\in\Z$ be a fundamental discriminant (i.e., either $d\equiv1\bmod4$ with $d$ square-free or $d=4d'$ with $d'\equiv2,3\bmod4$, with $d'$ square-free).
\begin{enumerate}
\item If $\gcd(N_E,d)=1$, then $W(E^d)=W(E)\cdot\left(\frac{d}{-N_E}\right)$ where $\left(\frac{\cdot}{\cdot}\right)$ is the Kronecker symbol.
\item If $d,d'$ are fundamental discriminants, relatively prime to $N_E$ and to each other then
\[W(E^{dd'})=W(E^d)W(E^{d'})W(E).\]
\end{enumerate}
\end{lemma}

We can extend part (2) of the previous lemma using induction, as follows.

\begin{corollary}\label{cor-rootnumberoftwistproduct}
Suppose $E$ is an elliptic curve over $\Q$, with conductor $N_E$. If $d_1,\ldots,d_n$ are fundamental discriminants, pairwise relatively prime and relatively prime to $N_E$, and $D=\prod_{k=1}^nd_k$, then
\[W(E^D)=W(E)^{n-1}\prod_{k=1}^nW(E^{d_k})=W(E)\prod_{k=1}^n\left(\frac{d_k}{-N_E}\right)=W(E)\left(\frac{D}{-N_E}\right).\]
\end{corollary}

\begin{proof}
Let $E$ and $d_1,\ldots,d_n,D$ be as in the statement. If $n=2$, then Lemma~\ref{kronlemma} shows that
\[W(E^D)=W(E)W(E^{d_1})W(E^{d_2})=W(E)^3\left(\frac{d_1}{-N_E}\right)\left(\frac{d_2}{-N_E}\right)=W(E)\left(\frac{D}{-N_E}\right).\]
Now suppose the result is settled for any $n>1$ values of $d_k$, and let $d_1,\ldots,d_{n+1}$ satisfy the same conditions, and put $D=d_1\cdots d_nd_{n+1}$ and $D'=d_1\cdots d_n$. Then:
\begin{align*}
W(E^D)&=W(E)W(E^{D'})W(E^{d_{n+1}})=W(E)\cdot W(E)^{n-1}\prod_{k=1}^nW(E^{d_k})\cdot W(E^{d_{n+1}})\\
&=W(E)^n\prod_{k=1}^{n+1}W(E^{d_k})=W(E)^{2n+1}\prod_{k=1}^{n+1}\left(\frac{d_k}{-N_E}\right)=W(E)\prod_{k=1}^{n+1}\left(\frac{d_k}{-N_E}\right)=W(E)\left(\frac{D}{-N_E}\right),
\end{align*}
which settles the induction step, and therefore the proof follows by induction.
\end{proof}

\subsection{Distribution of root numbers in quadratic twist families.} Next, we cite results of Rohrlich on the distribution of root numbers. We note that \cite[Theorem 1.1]{desjardins} extends it to the case $ab=0$: the quartic and sextic twist families.

\begin{theorem}[\cite{rohrlich}, Theorem 2]\label{thm-rohrlich-variation}
Consider the elliptic curve given by the equation $E:y^2=x^3+ax+b$ with $ab \ne 0$, and let $g(t)\in\Z[t],$ and the family of quadratic twists given by the equation
\[E^{g(t)}:y^2=x^3+ag(t)^2x+bg(t)^3.\]
Let $W_+=\{t\in\Q:g(t)\neq0,W(E^{g(t)})=1\}$ and define $W_-$ similarly. Then, one of the two properties holds:
\begin{enumerate}
\item The sets $W_+$ and $W_-$ are dense in $\R$.
\item The sets $W_+$ and $W_-$ are (in either order)
\[\{t\in\Q\mid g(t)<0\}\quad\text{and}\quad\{t\in\Q\mid g(t)>0\}.\]
\end{enumerate}
Moreover, there exists $g(t)$ such that (1) holds if and only if $E$ does not acquire everywhere good reduction over any abelian extension of $\Q$.
\end{theorem}

\begin{remark}\label{rem-bruteforce}
Let $E/\Q$ be an elliptic curve given by $y^2=f(x)$. Putting $g(t)=f(t)$ in Theorem~\ref{thm-rohrlich-variation}, one sees that the sets $W_\varepsilon$ are infinite (and dense in certain intervals), and therefore a ``brute force'' algorithm can easily find values $d$ equal to the square-free part of $f(t)$ for some appropriate values of $t\in\Q$ such that $E^d$ has the desired root number.
\end{remark}

Finally, we cite results on the existence of rank-0 and rank-1 twists. First we cite the known cases of the Birch and Swinnerton-Dyer conjecture.

\begin{theorem}[Gross--Zagier, Kolyvagin, Kolyvagin--Logachev \cite{grosszagier,kolyvagin,kolyvaginlogachev}]\label{thm-bsdrank01}
Let $E/\Q$ be an elliptic curve and let $L(E,s)$ be the Hasse--Weil $L$-function attached to $E$. Then:
\begin{enumerate}
\item If $\ord_{s=1}L(E,s)=0$, then $R_{E/\Q}=0$ and $\Sha(E/\Q)$ is finite.
\item If $\ord_{s=1}L(E,s)=1$, then $R_{E/\Q}=1$ and $\Sha(E/\Q)$ is finite,
\end{enumerate}
where $\Sha(E/\Q)$ denotes the Tate-Shafarevich group of $E$ over $\Q$.
\end{theorem}

Now we can state the theorems that show the existence of infinitely many quadratic twists of rank 0 and 1.

\begin{theorem}[Friedberg--Hoffstein, Kolyvagin--Logachev \cite{friedberghoffstein,kolyvaginlogachev}]\label{thm-twistsrank0}
Let $E/\Q$ be an elliptic curve. Then, there exist infinitely many fundamental discriminants $d$ such that
\[L(E^d,1)\neq0.\]
Hence, $R_{E^d/\Q}=0$ for infinitely many quadratic twists $E^d$ of $E$.
\end{theorem}

\begin{theorem}[Bump--Friedberg--Hoffstein, Murty--Murty, Gross--Zagier, and Kolyvagin \cite{BFH,murtymurty,grosszagier,kolyvagin}]\label{thm-twistsrank1}
Let $E/\Q$ be an elliptic curve. Then, there exist infinitely many fundamental discriminants $d$ such that
\[\ord_{s=1}L(E^d,s)=1.\]
Hence, $\rank E^d(\Q)=1$ for infinitely many quadratic twists $E^d$ of $E$.
\end{theorem}

To conclude this section, we give some background for and cite the sometimes-called $p$-parity theorem. For an elliptic curve $E/\Q$ and a prime $p$, let $\Sel_{p^\infty}(E/\Q)$ be the $p^\infty$-Selmer group of $E/\Q$, and let
\[s_p(E/\Q):=\corank_{\Z_p}\Sel_{p^\infty}(E/\Q)\]
be its $\Z_p$-corank, the \emph{$p^\infty$-Selmer rank} of $E/\Q$. The Kummer map and the definition of $\Sha(E/\Q)$ give the exact sequence
\begin{equation}\label{eq-kummer}
0\to E(\Q)\otimes\Q_p/\Z_p\to\Sel_{p^\infty}(E/\Q)\to\Sha(E/\Q)[p^\infty]\to0,
\end{equation}
and therefore
\begin{equation}\label{eq-selmer-rank}
s_p(E/\Q)=R_{E/\Q}+\corank_{\Z_p}\Sha(E/\Q)[p^\infty].
\end{equation}
In particular $s_p(E/\Q)\geq R_{E/\Q}$, with equality if $\Sha(E/\Q)[p^\infty]$ is finite.

\begin{theorem}[$p$-parity over $\Q$; Monsky \cite{monsky} for $p=2$, Dokchitser--Dokchitser \cite{DDannals} for all $p$]\label{thm-pparity}
Let $E/\Q$ be an elliptic curve and let $p$ be a prime. Then
\[(-1)^{s_p(E/\Q)}=W(E/\Q).\]
\end{theorem}

\begin{corollary}\label{cor-pparity-sha}
Let $E/\Q$ be an elliptic curve, and suppose that $\Sha(E/\Q)[p^\infty]$ is finite for some prime $p$. Then, $s_p(E/\Q)=R_{E/\Q}$ and the Parity Conjecture holds for $E$, that is, $W(E/\Q)=(-1)^{R_{E/\Q}}$.
\end{corollary}

\begin{proof}
If $\Sha(E/\Q)[p^\infty]$ is finite, its $\Z_p$-corank is 0, so $s_p(E/\Q)=R_{E/\Q}$ by Eq.~\eqref{eq-selmer-rank}. The result now follows from Theorem~\ref{thm-pparity}.
\end{proof}

\section{Proofs}\label{sec-proofs}

We are ready to present the proofs of the main theorems discussed in the introduction. We begin with a Corollary of Theorem~\ref{thm-fujita} that complements Theorem~\ref{thm-enrique} for multi-quadratic fields (it does not bound the number of such fields) and allows us to control the growth of torsion in multi-quadratic extensions.

\begin{corollary}\label{cor-fujita}
Let $E/\Q$ be an elliptic curve with conductor $N_E$. Then, there exists a finite multi-quadratic extension $F_E$ of $\Q$ such that if $L$ is another multi-quadratic extension of $\Q$ with $L\cap F_E=\Q$, then $E(L)_{tors}=E(\Q)_{tors}$. Moreover, $F_E=\Q(E[5040])\cap\Q(2)$ and $F_E$ is only ramified at primes dividing $2\cdot3\cdot5\cdot7\cdot N_E$.
\end{corollary}

\begin{proof}
Let $E/\Q$ be an elliptic curve and let $M\geq1$ be an integer such that any of the groups in Theorem~\ref{thm-fujita} is a subgroup of $\Z/M\Z\times\Z/M\Z$. One can in fact verify that $M=5040=16\cdot9\cdot5\cdot7$ works. Now let $F=\Q(E[M])$ and let $F(2)=F\cap\Q(2)$ be the maximal 2-elementary abelian extension of $\Q$ contained in $F$.

Now suppose $L$ is a multi-quadratic extension of $\Q$ (i.e., $L\subseteq\Q(2)$), and suppose that $L\cap F(2)=\Q$. Since $L$ itself is multi-quadratic, this also implies that $L\cap F=\Q$. Then, we claim that $E(L)_\tors=E(\Q)_\tors$. Indeed, by Theorem~\ref{thm-fujita} we must have that $E(L)_\tors\subseteq E(\Q(2))_\tors\subseteq E[M]$ by construction of $M$. Suppose for a contradiction that $E(\Q)_\tors\subsetneq E(L)_\tors$. Then, $P\in E[M]$ such that $P\notin E(\Q)$ but $P\in E(L)$, and so $\Q\subsetneq\Q(P)\subset L$. Thus, $\Q(P)\subseteq L\cap\Q(E[M])$. But we have assumed that $L\cap F=L\cap\Q(E[M])$ is trivial, and therefore $\Q(P)=\Q$, a contradiction.

Hence, we have shown that $E(L)_\tors=E(\Q)_\tors$, with $F_E=F(2)$. Note that by the criterion of N\'eron--Ogg--Shafarevich, the only primes that may ramify in $\Q(E[M])$ are divisors of $M\cdot N_E$, and therefore $F_E\subseteq\Q(E[M])$ can only ramify at primes dividing $2\cdot3\cdot5\cdot7\cdot N_E$ as claimed. This concludes the proof.
\end{proof}

We remark here that with more refined information about the $\ell$-adic images attached to the curve $E/\Q$ for $\ell=2,3,5,7$, one can effectively compute the field $F_E$ in Corollary~\ref{cor-fujita}.

Next, we prove the key lemma that the main results are deduced from.

\begin{lemma}\label{lem:multiquadratic-rank}
Let $E/\Q$ be an elliptic curve with conductor $N_E$, global root number $W(E)$ and rank $R_{E/\Q}$. Let $d_1,\ldots,d_n$ be square-free integers such that $L_n=\Q(\sqrt{d_1},\ldots,\sqrt{d_n})$ has degree $2^n$ over $\Q$, let $\mathcal{D}_n$ be as in Proposition~\ref{prop-kida}, and suppose that $R_{E^{d_k}/\Q}\geq1$ for every $k=1,\ldots,n$. Then:
\begin{enumerate}
\item $R_{E/L_n}\geq R_{E/\Q}+n$.
\end{enumerate}
Moreover, suppose that the discriminant $D_k$ of $K_k=\Q(\sqrt{d_k})$ is relatively prime to $N_E$ for every $k=1,\ldots,n$. Then:
\begin{enumerate}\setcounter{enumi}{1}
\item If $W(E)=-1$ and $W(E^{d_k})=1$ for every $k$, and the Parity Conjecture holds for $E^d$ for every $d\in\mathcal{D}_n$ with $d\neq1$, then
\[R_{E/L_n}\geq R_{E/\Q}+2n+2^{n-1}-1.\]
\item If $W(E)=-1$ and $W(E^{d_k})=-1$ for every $k$, then $W(E^d)=-1$ for every $d\in\mathcal{D}_n$, and either
\begin{enumerate}
\item $\ord_{s=1}L(E^d,s)=1$ for every $d\in\mathcal{D}_n$ with $d\neq1$, in which case
\[R_{E/L_n}=R_{E/\Q}+2^n-1,\]
or
\item $\ord_{s=1}L(E^d,s)>1$ for some $d\in\mathcal{D}_n$ with $d\neq1$. In this case, if the Parity Conjecture holds for $E^d$ for every $d\in\mathcal{D}_n$ with $d\neq1$, or if for every such $d$ there is a prime $p$ such that $\Sha(E^d/\Q)[p^\infty]$ is finite, then
\[R_{E/L_n}\geq R_{E/\Q}+2^n-1.\]
If BSD holds, then in fact $R_{E/L_n}\geq R_{E/\Q}+2^n+1$.
\end{enumerate}
\item Suppose $W(E)=1$. If the Parity Conjecture holds for $E^d$ for every $d\in\mathcal{D}_n$ with $d\neq1$, or if for every such $d$ there is a prime $p$ such that $\Sha(E^d/\Q)[p^\infty]$ is finite, then
\begin{enumerate}
\item if $W(E^{d_k})=1$ for every $k$, we have $R_{E/L_n}\geq R_{E/\Q}+2n$,
\item if $W(E^{d_k})=-1$ for every $k$, we have $R_{E/L_n}\geq R_{E/\Q}+2^{n-1}$, and
\item if $n\geq2$, $W(E^{d_k})=1$ for $1\leq k\leq n-1$ and $W(E^{d_n})=-1$, we have $R_{E/L_n}\geq R_{E/\Q}+2^{n-1}+2n-2$.
\end{enumerate}
\end{enumerate}
\end{lemma}

\begin{proof}
By Proposition~\ref{prop-kida},
\begin{equation}\label{eq-rank-sum}
R_{E/L_n}=\sum_{d\in\mathcal{D}_n}R_{E^d/\Q}=R_{E/\Q}+\sum_{d\in\mathcal{D}_n,\,d\neq1}R_{E^d/\Q},
\end{equation}
and all terms are non-negative. Since $[L_n:\Q]=2^n$, the elements of $\mathcal{D}_n$ represent $2^n$ distinct classes in $\Q^\times/(\Q^\times)^2$, and $d_1,\ldots,d_n$ are among them. Hence $R_{E/L_n}\geq R_{E/\Q}+\sum_{k=1}^nR_{E^{d_k}/\Q}\geq R_{E/\Q}+n$, which proves (1).

Now assume that $\gcd(D_k,N_E)=1$ for every $k$. We first compute the root numbers of all the twists $E^d$ with $d\in\mathcal{D}_n$. Let $d=\prod_{k\in S}d_k$ for some $S\subseteq\{1,\ldots,n\}$, and let $D$ be the discriminant of $\Q(\sqrt{d})$, so that $E^d\cong E^D$ over $\Q$. Then $\prod_{k\in S}D_k=D\,m^2$ for some integer $m$ whose prime factors divide $\prod_{k\in S}D_k$. In particular $D$ and $m$ are relatively prime to $N_E$. By part (1) of Lemma~\ref{kronlemma}, and since the Kronecker symbol $\left(\frac{\cdot}{-N_E}\right)$ is completely multiplicative in the numerator,
\begin{equation}\label{eq:rootnumber-product}
W(E^d)=W(E)\left(\frac{D}{-N_E}\right)=W(E)\left(\frac{Dm^2}{-N_E}\right)=W(E)\prod_{k\in S}\left(\frac{D_k}{-N_E}\right)=W(E)^{|S|+1}\prod_{k\in S}W(E^{d_k}).
\end{equation}
When the $D_k$ are pairwise relatively prime, this is Corollary~\ref{cor-rootnumberoftwistproduct}.

Now suppose that $W(E)=-1$ and $W(E^{d_k})=1$ for every $k$, as in (2). Then Eq.~\eqref{eq:rootnumber-product} gives $W(E^d)=(-1)^{|S|+1}$. By Conjecture~\ref{conj-parity} applied to each $E^d$ with $d\neq 1$:
\begin{itemize}
\item if $|S|=1$, then $R_{E^{d_k}/\Q}$ is even and, by hypothesis, at least 1, so $R_{E^{d_k}/\Q}\geq2$;
\item if $|S|\geq2$ is even, then $W(E^d)=-1$, so $R_{E^d/\Q}$ is odd and $R_{E^d/\Q}\geq1$. A simple combinatorial count shows that there are $2^{n-1}-1$ such subsets $S$;
\item if $|S|\geq3$ is odd, we only use $R_{E^d/\Q}\geq0$.
\end{itemize}
Substituting into Eq.~\eqref{eq-rank-sum}, we obtain $R_{E/L_n}\geq R_{E/\Q}+2n+2^{n-1}-1$, which proves (2).

To show (3), suppose that $W(E)=W(E^{d_k})=-1$ for every $k$. Then Eq.~\eqref{eq:rootnumber-product} gives $W(E^d)=-1$ for every $d\in\mathcal{D}_n$. Since $W(E^d)$ is the sign of the functional equation of $L(E^d,s)$, the order $\ord_{s=1}L(E^d,s)$ is odd, hence at least 1, for every $d\in\mathcal{D}_n$.

If $\ord_{s=1}L(E^d,s)=1$ for every $d\in\mathcal{D}_n$ with $d\neq1$, then part (2) of Theorem~\ref{thm-bsdrank01} gives $R_{E^d/\Q}=1$ for each of these $2^n-1$ twists. Thus $R_{E/L_n}=R_{E/\Q}+2^n-1$. This shows (3)(a).

For (3)(b), suppose first that the Parity Conjecture holds for every $E^d$ with $d\in\mathcal{D}_n$, $d\neq1$. Then for each such value $d\neq1$ in $\mathcal{D}_n$ we can conclude $R_{E^d/\Q}$ is odd, hence at least 1, and Eq.~\eqref{eq-rank-sum} gives $R_{E/L_n}\geq R_{E/\Q}+2^n-1$. Otherwise, if instead we have that for each $d\in\mathcal{D}_n$ with $d\neq1$ there is a prime $p$ such that $\Sha(E^d/\Q)[p^\infty]$ is finite, then Eq.~\eqref{eq-selmer-rank} gives $R_{E^d/\Q}=s_p(E^d/\Q)$, and the $p$-parity theorem (Theorem~\ref{thm-pparity}; see also Cor.~\ref{cor-pparity-sha}) gives $(-1)^{s_p(E^d/\Q)}=W(E^d)=-1$. So again $R_{E^d/\Q}$ is odd and at least 1, and the same bound follows. Finally, if BSD holds, then $R_{E^d/\Q}=\ord_{s=1}L(E^d,s)$ for every $d\neq1$. Each of these is odd and positive, and it is at least 3 for some $d_0\neq1$ by hypothesis. Hence
\[R_{E/L_n}\geq R_{E/\Q}+(2^n-2)+3=R_{E/\Q}+2^n+1.\]
This proves (3)(b).

Finally, to show (4), suppose that $W(E)=1$. In all three cases, for $d\in\mathcal{D}_n$ with $d\neq1$, the hypotheses imply that $(-1)^{R_{E^d/\Q}}=W(E^d)$: under Conjecture \ref{conj-parity} this is the conjecture itself, and if $\Sha(E^d/\Q)[p^\infty]$ is finite this is Corollary~\ref{cor-pparity-sha}.

For (4)(a), suppose $W(E^{d_k})=1$ for every $k$. Then Eq.~\eqref{eq:rootnumber-product} gives $W(E^d)=1$ for every $d\in\mathcal{D}_n$. In particular, each $R_{E^{d_k}/\Q}$ is even and, by hypothesis, at least 1, so $R_{E^{d_k}/\Q}\geq 2$ for $k=1,\ldots,n$. Using $R_{E^d/\Q}\geq0$ for the remaining $d$, Eq.~\eqref{eq-rank-sum} gives $R_{E/L_n}\geq R_{E/\Q}+2n$.

For (4)(b), suppose $W(E^{d_k})=-1$ for every $k$. Then Eq.~\eqref{eq:rootnumber-product} gives $W(E^d)=(-1)^{|S|}$ for $d=\prod_{k\in S}d_k$. Let $d\in\mathcal{D}_n$ with $|S|$ odd, so that $W(E^d)=-1$. Then $R_{E^d/\Q}$ is odd, hence at least 1. There are $2^{n-1}$ subsets $S\subseteq\{1,\ldots,n\}$ of odd size. If $|S|\geq2$ is even, then $W(E^d)=1$, and we only use $R_{E^d/\Q}\geq0$. Substituting into Eq.~\eqref{eq-rank-sum}, we obtain $R_{E/L_n}\geq R_{E/\Q}+2^{n-1}$, which proves (4)(b).

For (4)(c), suppose $n\geq2$, with $W(E^{d_k})=1$ for $k<n$ and $W(E^{d_n})=-1$. Since $W(E)=1$, Eq.~\eqref{eq:rootnumber-product} gives $W(E^d)=\prod_{k\in S}W(E^{d_k})$, which equals $-1$ if $n\in S$ and $+1$ if $n\notin S$, for $d=\prod_{k\in S}d_k$. There are $2^{n-1}$ subsets $S\subseteq\{1,\ldots,n\}$ with $n\in S$, all of them non-empty, and for each of them $R_{E^d/\Q}$ is odd, hence at least $1$. The $n-1$ singletons $S=\{k\}$ with $k<n$ do not contain $n$, so they are not among those subsets. For each of them $W(E^{d_k})=1$, so $R_{E^{d_k}/\Q}$ is even and, by hypothesis, at least $1$, hence $R_{E^{d_k}/\Q}\geq2$. Using $R_{E^d/\Q}\geq0$ for the remaining $d$, Eq.~\eqref{eq-rank-sum} gives $R_{E/L_n}\geq R_{E/\Q}+2^{n-1}+2(n-1)$. This proves (4) and concludes the proof of the lemma.
\end{proof}

\begin{remark}\label{rem-bsd}
In (3)(b), the Parity Conjecture (or the finiteness of $\Sha(E^d/\Q)[p^\infty]$) only shows that $R_{E^{d_0}/\Q}$ is odd. It does not rule out $R_{E^{d_0}/\Q}=1$ when $\ord_{s=1}L(E^{d_0},s)\geq3$, which is why the strict inequality requires BSD.
\end{remark}

Now our task is to construct twists of curves that fit the recipe laid out in Lemma~\ref{lem:multiquadratic-rank}.

\begin{lemma}\label{findgoodvaluelemma}
Let $f(x)=x^3+Ax+B$ with $A,B\in\Z$, let $N=p_1\cdots p_n$ be a product of distinct primes $p_1<p_2<\ldots<p_n$, let $q=q_1\cdots q_m$ be a product of distinct odd primes not dividing $N$, with $m\geq1$, and let $x_0=x_1/(qN)^2$ with $x_1\in\Z$ such that
\[x_1\equiv\begin{cases}1\bmod 8&\text{if }p_1=2,\\ s_i\bmod p_i&\text{if }p_i>2,\\ t_j\bmod q_j,\end{cases}\]
where, for $1\leq i\leq n$ and $1\leq j\leq m$, each $s_i$ is a fixed nonzero quadratic residue mod $p_i$, and $t_j$ is a fixed quadratic non-residue mod $q_j$. Let $d_0$ be the square-free part of $f(x_0)$, that is, $f(x_0)=d_0\cdot a^2$, with $d_0\in\Z$ and some $a\in\Q$. Then, the Kronecker symbol of $d_0$ satisfies
\[\left(\frac{d_0}{p}\right)=\begin{cases}1&\text{for }p=p_i,\text{ and }1\leq i\leq n,\\-1&\text{for }p=q_j.\end{cases}\]
Moreover, $d_0$ is not a square in $\Q$, the integer $x_1$ can be chosen so that $d_0$ is either positive or negative, and if $N$ is even, then $d_0\equiv1\bmod 8$. More precisely, if $|x_1|\geq(|A|+|B|+1)(qN)^2$, then $\sign(d_0)=\sign(x_1)$.
\end{lemma}

\begin{proof}
Let $f(x)=x^3+Ax+B$ with $A, B\in\Z$. Consider $x_0=x_1/(qN)^2$, for some $x_1\in\Z$ satisfying the congruence conditions of the statement of the lemma, and let $d_0$ be the square-free part of $f(x_0)$. Then,
\[f(x_0)=f(x_1/(qN)^2)=\frac{x_1^3}{(q^2N^2)^3}+\frac{Ax_1}{q^2N^2}+B=\frac{x_1^3+q^4N^4Ax_1+q^6N^6B}{q^6N^6}\]
Thus, $f(x_0)=(1/(qN)^3)^2\cdot(x_1^3+q^4N^4Ax_1+q^6N^6B)$ and, in particular, if $x_1\equiv1\bmod 8$ and $x_1\equiv s_i\bmod p_i$, then
\[x_1^3+q^4N^4Ax_1+q^6N^6B\equiv1\bmod 8,\quad\equiv s_i^3\bmod p_i,\quad\text{and}\quad\equiv t_j^3\bmod q_j.\]
We note here that $d_0$ is therefore equivalent to $x_1^3+q^4N^4Ax_1+q^6N^6B$ modulo squares, so if $2|N$, then $x_1\equiv1\bmod8$ and we conclude that $d_0\equiv1\bmod8$. Furthermore, the Kronecker symbol of $d_0$, the square-free part of $f(x_0)$, has value
\[\left(\frac{d_0}{p}\right)=\left(\frac{x_1^3+q^4N^4Ax_1+q^6N^6B}{p}\right)=1,\]
if $p=p_i$, for $1\leq i\leq n$, and the symbol equals $-1$ if $p=q_j$, as desired. In particular, since $m\geq1$, we have $\left(\frac{d_0}{q_1}\right)=-1$, so $d_0$ is not a square.

Finally, we note that since $f(x)$ is a cubic polynomial, one can choose $x_1$ to satisfy all the necessary congruences and be positive and sufficiently large so that $f(x)$ and $d_0$ are also positive or, conversely, $x_1$ can be chosen to satisfy the necessary congruences and be negative and $|x_1|$ sufficiently large so that $d_0$ is negative. Explicitly, if $|x_1|\geq(|A|+|B|+1)(qN)^2$, then $|x_0|\geq|A|+|B|+1\geq1$, so $|Ax_0+B|\leq(|A|+|B|)|x_0|<|x_0|^2\leq|x_0|^3$, and therefore $\sign(d_0)=\sign(f(x_0))=\sign(x_0)=\sign(x_1)$.
\end{proof}

\begin{lemma}\label{extnlemma}
Let $f(x)=x^3+Ax+B$ with $A,B\in\Z$, let $N=p_1\ldots p_g$ be a product of distinct primes $p_1<p_2<\ldots<p_g$, and let $\varepsilon_k\in\{\pm1\}$, for $k\geq 1$, be a fixed sequence of signs. Then, there are infinitely many rational numbers $z_1,z_2,\ldots$ such that
\begin{enumerate}
\item $K_k=\Q(\sqrt{f(z_k)})$ is a quadratic extension of $\Q$, and $K_k\neq\Q(i)$, for all $k\geq1$,
\item for any integer $n\geq1$, the compositum $L_n=\prod_{k=1}^nK_k$ is an extension of degree $2^n$ over $\Q$,
\item if $d_i$ is the square-free part of $f(z_i)$, then $\left(\frac{d_i}{p}\right)=1$ for every prime $p$ dividing $2N$, and $d_i\equiv1\bmod 8$,
\item the numbers $d_1,\ldots,d_n$ are pairwise relatively prime and also relatively prime to $2N$,
\item the numbers $z_1,z_2,\ldots$ can be chosen so that $\sign(d_k)=\varepsilon_k$ for every $k\geq1$.
\end{enumerate}
\end{lemma}

\begin{proof}
Let $N^*=N$ if $N$ is even, and $2N$ if $N$ is odd. We will construct the rational numbers $z_k$, for $k\geq1$, recursively. Let $q_1$ be the smallest odd prime not dividing $N^*$. Let $z_1=x_0$ where $x_0$ is a rational value output of Lemma~\ref{findgoodvaluelemma} for $N^*$ and $q=q_1$, with $s_i=1$ for $1\leq i\leq g$, and $t_1$ a fixed quadratic non-residue mod $q_1$. Let $d_1=\sqf(f(z_1))$, the square-free part of $f(z_1)$, which Lemma~\ref{findgoodvaluelemma} shows is not a square, i.e., $d_1\neq1$. Thus, $K_1=\Q(\sqrt{f(z_1)})$ is a quadratic extension of $\Q$, and $\left(\frac{d_1}{p}\right)=1$ for every prime $p$ dividing $N^*$ (and therefore every prime dividing $2N$), again by the lemma. Finally, note that 2 divides $N^*$, so the square-free part satisfies $d_1\equiv x_1^3\equiv1\bmod 8$, and therefore $d_1\not\equiv-1\bmod(\Q^\times)^2$. Hence, $K_1\neq\Q(i)$.

Now suppose we have defined $z_1,\ldots,z_k$, with $K_i$, and $d_i=\sqf(f(z_i))$ satisfying the desired properties, for every $1\leq i\leq k$. Let $q=q_1\prod_{i=1}^k\prod_{p|d_i}p=q_1q_2\ldots q_m$. Notice that by construction each prime $q_j$ dividing $q$ is relatively prime to $N$, because of the values of the Kronecker symbol of $d_i$ at primes dividing $N$. Now let $z_{k+1}=x_0=x_1/(qN^*)^2$ be the output of Lemma~\ref{findgoodvaluelemma} for $N^*$ and $q$, with $x_1\equiv1\bmod 8$, and $s_i\equiv1\bmod p_i$ for $1\leq i\leq n$, and $t_j$ a set of fixed quadratic non-residues mod $q_j$ for $1\leq j\leq m$, and let $d_{k+1}=\sqf(f(z_{k+1}))$ and $K_{k+1}=\Q(\sqrt{f(z_{k+1})})=\Q(\sqrt{d_{k+1}})$.

Since $d_{k+1}$ is not a square, $K_{k+1}/\Q$ is quadratic, and it satisfies (1). By Lemma~\ref{findgoodvaluelemma}, the value $d_{k+1}$ satisfies (3). Moreover, by construction $K_{k+1}\neq\Q(i)$. In particular, $d_{k+1}$ is not $\pm1$ modulo squares, and $d_{k+1}\equiv1\bmod8$. Therefore, there is an odd prime divisor $\ell$ of $d_{k+1}$. Further, $d_{k+1}$ divides $x_1^3+q^4N^4Ax_1+q^6N^6B$, which is relatively prime to $qN$ by the congruence conditions on the integer $x_1$, and therefore the odd prime divisor $\ell$ of $d_{k+1}$ ramifies in $K_{k+1}$ but it does not ramify in the compositum $K_1\cdots K_k$, because $\gcd(\ell,2\prod_{i=1}^kd_i)=1$. Hence, they are disjoint extensions, and the degree of $L_{k+1}=\prod_{i=1}^{k+1}K_{i}$ must be $2^k\cdot2=2^{k+1}$, and so (2) is also verified.

It also follows from the above discussion that $d_{k+1}$ is constructed (via congruence restrictions) to be relatively prime to any prime $p|N^*$ or $q|\prod_{i=1}^kd_k$. This shows (4).

Finally, fix a sequence of signs $\varepsilon_k\in\{\pm1\}$. Lemma~\ref{findgoodvaluelemma} shows that each $z_k$ can be chosen so that $\sign(d_k)=\varepsilon_k$, for every $k\geq1$: at step $k$ the sign of $x_1$ is chosen independently of the previous steps, and the congruence conditions on $x_1$ do not depend on its sign. This shows (5) and concludes the proof of the lemma.
\end{proof}

\begin{theorem}\label{maintheorem}
Suppose $E/\Q$ is an elliptic curve with rank $R_{E/\Q}$ and conductor $N_E$. Fix a sequence of signs $\varepsilon_k\in\{\pm1\}$, $k\geq1$. Then, there are infinitely many distinct square-free integers $d_k\neq\pm1$, one for each $k\geq1$, that can be explicitly recursively constructed, such that
\begin{enumerate}
\item the numbers $d_1,d_2,\ldots$ are pairwise relatively prime and also relatively prime to $2N_E$,
\item for any integer $n\geq1$, the compositum $L_n=\prod_{k=1}^nK_k$ is an extension of degree $2^n$ over $\Q$, where $K_k=\Q(\sqrt{d_k})$, and
\item the quadratic twist $E^{d_k}$ has $R_{E^{d_k}/\Q}\geq1$ and root number $W(E^{d_k})=\varepsilon_k$.
\item for each $n\geq1$, we have $E(L_n)_{tors}=E(\Q)_{tors}$.
\end{enumerate}
\end{theorem}

\begin{proof}
Fix a sequence of signs $\varepsilon_k\in\{\pm1\}$. Let $E/\Q$ be an elliptic curve, and let $y^2=f(x)=x^3+Ax+B$, with $A,B\in\Z$, be a short Weierstrass model for $E$. Let
\[N=\begin{cases}\prod_{p|N_E}p&\text{if }\nu_2(N_E)>0\\ 2\cdot\prod_{p|N_E}p&\text{if }\nu_2(N_E)=0,\end{cases}\]
and let $W(E)$ be the global root number of $E/\Q$. By Lemma~\ref{extnlemma}, there exists an infinite sequence of rational numbers $z_k$ with square-free part $d_k$ that satisfy (1)-(5) in the Lemma, with $\sign(d_k)=\varepsilon_k\cdot W(E)$ for every $k\geq1$. Lemma~\ref{extnlemma}, parts (2) and (4), show parts (1) and (2) in the theorem we are currently proving.

To show (3), note that by construction $\left(\frac{d_k}{p}\right)=1$ for all $p$ dividing $N$, and $d_k\equiv1\bmod 8$ (so $d_k\equiv1\bmod4$ and it is a fundamental discriminant). Now using Lemma~\ref{kronlemma},
\[W(E^{d_k})=\left(\frac{d_k}{-N_E}\right)W(E)=\left(\frac{d_k}{-1}\right)\prod_{p|N_E}\left(\frac{d_k}{p}\right)^{\nu_p(N_E)}W(E)=\left(\frac{d_k}{-1}\right)W(E)=\varepsilon_k\cdot W(E)^2=\varepsilon_k,\]
where we have used the fact that $\left(\frac{d_k}{-1}\right)=\sign(d_k)=\varepsilon_k\cdot W(E)$. Thus, $W(E^{d_k})=\varepsilon_k$ for all $k\geq1$. Furthermore, let $K_k=\Q(\sqrt{d_k})$. By Theorem~\ref{thm-najmantwists}, $E(\Q)[2]\subsetneq E^d(\Q)_\tors$ for at most three classes $d\in\Q^\times/(\Q^\times)^2$, and the $d_k$ are pairwise distinct modulo squares by part (2) of Lemma~\ref{extnlemma}. Hence, after discarding at most three of the $d_k$ (and relabeling), we may assume that $E^{d_k}(\Q)_\tors=E(\Q)[2]$ for every $k\geq1$.

Now, if we let $f(z_k)=d_k\cdot m_k^2$ (where $m_k$ is a non-zero rational number), then the point $P_k=(z_k,m_k)$ belongs to the quadratic twist $E^{d_k}:d_ky^2=f(x)$, and since $m_k\neq0$ (because $\Q(\sqrt{f(z_k)})/\Q$ is a quadratic extension), the point $P_k$ is not a 2-torsion point on $E^{d_k}$. The fact that $E^{d_k}(\Q)_\tors=E(\Q)[2]$ now implies that $P_k$ is a point of infinite order on $E^{d_k}$ and therefore $R_{E^{d_k}/\Q}\geq1$, and $W(E^{d_k})=\varepsilon_k$. This shows (3). On the model $y^2=x^3+d_k^2Ax+d_k^3B$, the point $P_k$ corresponds to $(d_kz_k,d_k^2m_k)$.

Finally, let $F_E$ be the finite multi-quadratic extension of $\Q$ given by Cor.~\ref{cor-fujita}, and let $f_1,\ldots,f_m$ be a complete list of square-free integer divisors of $2\cdot3\cdot5\cdot7\cdot N_E=210N_E$. Since $\{d_k\}_{k\geq1}$ is infinite, we can pick a new subsequence $\{d_k'\}_{k\geq1}$ as follows:
\begin{enumerate}
\item Pick $d_1'\neq \pm f_j$ for any $f_1,\ldots,f_m$. This ensures that $K_1=\Q(\sqrt{d_1'})$ is disjoint from $F_E$ and therefore $E(K_1)_\tors=E(\Q)_\tors$.
\item If $d_1',\ldots,d_n'$ have been chosen, so that $L_n\cap F_E=\Q$, choose a new value $d_{n+1}'$ among the remaining values of $\{d_k\}$ such that $d_{n+1}'$ is divisible by a prime $\ell$ not dividing $C_N=210\cdot d_1'\cdots d_n'\cdot N_E$. Since there are only finitely many square-free divisors of $C_N$, such $d_{n+1}'$ must exist, and $\Q(\sqrt{d_{n+1}'})$ ramified at a prime $\ell$ that does not ramify in $L_n/\Q$ or in $\Q(E[M])$. Hence $L_{n+1}=L_n(\sqrt{d_{n+1}'})$ is of degree $2^{n+1}$ and $L_{n+1}\cap\Q(E[M])=\Q$, and therefore $E(L_{n+1})_\tors=E(\Q)_\tors$, as desired.
\end{enumerate}
When the signs $\varepsilon_k$ are not all equal, the discarding and relabeling above, and the choice of the subsequence $\{d_k'\}$, must respect the prescribed signs. To arrange this, apply Lemma~\ref{extnlemma} with a sequence of signs in which both $+1$ and $-1$ occur infinitely often, so that infinitely many values $d_k$ with $W(E^{d_k})=+1$, and infinitely many with $W(E^{d_k})=-1$, are available. At each step above only finitely many values are excluded, so we may choose $d'_{n+1}$ (and $d'_1$) among the remaining values with $W(E^{d'_{n+1}})=\varepsilon_{n+1}$. Then $W(E^{d'_k})=\varepsilon_k$ for every $k$, and the rest of the argument is unchanged.
This concludes the proof.
\end{proof}

\begin{remark}[A factoring-free variant]\label{rem-factoringfree}
The recursion in the proof of Lemma~\ref{extnlemma} uses $q=q_1\prod_{i\leq k}\prod_{p\mid d_i}p$, which requires the prime factorizations of $d_1,\ldots,d_k$. For the curves of Section~\ref{sec-examples} these integers have more than a thousand digits, so this is not effective in practice. The following variant needs no factorization, and it is the one implemented in Section~\ref{sec-code}. Let $q_k$ be the $k$-th odd prime not dividing $N^*$, let $t_k$ be a quadratic non-residue modulo $q_k$, and for each $k\geq1$ let $z_k=x_1/(q_1\cdots q_kN^*)^2$, where
\[x_1\equiv1\bmod 8,\quad x_1\equiv1\bmod p\ (p\mid N^*\text{ odd}),\quad x_1\equiv1\bmod q_j\ (j<k),\quad x_1\equiv t_k\bmod q_k.\]
This is Lemma~\ref{findgoodvaluelemma} applied to $N^*q_1\cdots q_{k-1}$ and $q=q_k$, so $d_k=\sqf(f(z_k))$ satisfies $d_k\equiv1\bmod 8$, $\left(\frac{d_k}{p}\right)=1$ for $p\mid N^*$, $\left(\frac{d_k}{q_j}\right)=1$ for $j<k$, and $\left(\frac{d_k}{q_k}\right)=-1$. If $S\subseteq\{1,\ldots,n\}$ is non-empty and $s=\min S$, then
\[\left(\frac{\prod_{k\in S}d_k}{q_s}\right)=\left(\frac{d_s}{q_s}\right)\prod_{k\in S,\,k>s}\left(\frac{d_k}{q_s}\right)=-1,\]
so no such product is a square, and $[L_n:\Q]=2^n$. These properties do not depend on the sign of $x_1$, which is chosen separately at each step, so the sign of each $d_k$, and hence $W(E^{d_k})$, can be prescribed independently for every $k$. Thus parts (1), (2), (3) and (5) of Lemma~\ref{extnlemma} hold, but (4) is replaced by the weaker statement that each $d_k$ is relatively prime to $2N$. Consequently, Theorem~\ref{maintheorem} holds for this variant with part (1) replaced by ``each $d_k$ is relatively prime to $2N_E$'': the proof of (2) and (3) only uses $\gcd(d_k,2N_E)=1$, and so does Lemma~\ref{lem:multiquadratic-rank}, which needs $\gcd(D_k,N_E)=1$ but not pairwise coprimality. For (4), the choice of the subsequence $\{d'_k\}$ above does not use pairwise coprimality either. Moreover, if $N^*$ is enlarged to $N=2\cdot3\cdot5\cdot7\cdot\prod_{p\mid N_E,\,p>7}p$, as in Section~\ref{sec-code}, no subsequence is needed: for every non-empty $S$, the square-free part of $\prod_{k\in S}d_k$ is $\equiv1\bmod 8$, prime to $210N_E$ and not a square, hence not $\pm1$, so it has a prime factor $\ell\nmid 210N_E$, and $\Q(\sqrt{\prod_{k\in S}d_k})$ is ramified at $\ell$ and not contained in $F_E$. Therefore $L_n\cap F_E=\Q$ and $E(L_n)_\tors=E(\Q)_\tors$ by Corollary~\ref{cor-fujita}.
\end{remark}

\begin{remark}\label{rem-signpatterns}
The bound in Lemma~\ref{lem:multiquadratic-rank}(4)(c) is at least the bound in (4)(a) for $n\geq2$, with equality for $n=2$, and it is larger than the bound in (4)(b) for every $n\geq2$. More generally, suppose that $\gcd(D_k,N_E)=1$ for every $k$, that $W(E^{d_k})=\varepsilon_k$ for an arbitrary choice of signs, and let $P=\{k:\varepsilon_k=1\}$ and $p=|P|$. Counting as in the proof of Lemma~\ref{lem:multiquadratic-rank} (rank $\geq1$ for each $d\neq1$ with $W(E^d)=-1$, and rank $\geq2$ for each $d_k$ with $k\in P$) gives:
\begin{itemize}
\item if $W(E)=1$, then $W(E^d)=-1$ if and only if $|S\setminus P|$ is odd, and the count is $2n$ if $p=n$ and $2^{n-1}+2p$ if $p<n$;
\item if $W(E)=-1$, then $W(E^d)=-1$ if and only if $|S\cap P|$ is even, and the count is $2^n-1$ if $p=0$ and $2^{n-1}-1+2p$ if $p\geq1$.
\end{itemize}
Hence, among all choices of signs, the largest lower bound obtained in this way is $\max\{2n,2^{n-1}+2n-2\}$ if $W(E)=1$, attained by Lemma~\ref{lem:multiquadratic-rank}(4)(a) and (4)(c), and $\max\{2n+2^{n-1}-1,2^n-1\}$ if $W(E)=-1$, attained by (2) and (3). In particular, mixing signs does not improve the bounds when $W(E)=-1$.
\end{remark}

\subsection{Proof of Theorem~\ref{thm-main2}.} The proof of Theorem~\ref{thm-main2} is now immediate. Theorem~\ref{maintheorem} provides the sequence $\{d_k\}_{k\geq1}$ with the required properties, and then Lemma~\ref{lem:multiquadratic-rank} proves the properties exactly as stated.

\subsection{Proof of Theorem~\ref{thm-main1}.} Fix $T$, $\ast\in\{1,\infty\}$ and $n\geq1$. If $\sup(\Rk^\ast_{1,T})=\infty$, then for every $r$ there is a curve $E/\Q$ with $E(\Q)_\tors\cong T$ and $\rank E(\Q)\geq r$ (infinitely many non-isomorphic ones if $\ast=\infty$), and by Theorem~\ref{maintheorem}(4) and Lemma~\ref{lem:multiquadratic-rank}(1) the base change of each of them to the corresponding field $L_n$ of degree $2^n$ has torsion $T$ and rank $\geq r+n$. So $\sup(\Rk^\ast_{2^n,T})=\infty$ as well, and there is nothing to prove. Assume now that $R=\sup(\Rk^\ast_{1,T})$ is finite. Since $\Rk^\ast_{1,T}$ is a set of non-negative integers, the supremum is attained. If $\ast=1$, choose one elliptic curve $E/\Q$ with $E(\Q)_\tors\cong T$ and $\rank E(\Q)=R$, and put $I=\{1\}$, $E_1=E$. If $\ast=\infty$, by definition of $\Rk^\infty_{1,T}$ there are infinitely many elliptic curves $\{E_i\}_{i\in I}$ over $\Q$, pairwise non-isomorphic over $\overline{\Q}$, with $E_i(\Q)_\tors\cong T$ and $\rank E_i(\Q)=R$.

For each $i\in I$, the Parity Conjecture gives $W(E_i)=(-1)^R$. Apply Theorem~\ref{maintheorem} to $E_i$, with a choice of signs $\varepsilon_k$ to be made below, and let $L_{n,i}$ be the resulting field of degree $2^n$. By Theorem~\ref{maintheorem}(4), $E_i(L_{n,i})_\tors=E_i(\Q)_\tors\cong T$. By Theorem~\ref{maintheorem}(1), the discriminants of the fields $\Q(\sqrt{d_k})$ are prime to $N_{E_i}$, so Lemma~\ref{lem:multiquadratic-rank} applies, and Parity holds for all twists by assumption:
\begin{itemize}
\item If $R$ is odd, then $W(E_i)=-1$. With all $\varepsilon_k=+1$, Lemma~\ref{lem:multiquadratic-rank}(2) gives $\rank E_i(L_{n,i})\geq R+2n+2^{n-1}-1$; with all $\varepsilon_k=-1$, Lemma~\ref{lem:multiquadratic-rank}(3) gives $\rank E_i(L_{n,i})\geq R+2^n-1$ (in case (a) with equality, in case (b) under Parity).
\item If $R$ is even, then $W(E_i)=+1$. With all $\varepsilon_k=+1$, Lemma~\ref{lem:multiquadratic-rank}(4)(a) gives $\rank E_i(L_{n,i})\geq R+2n$; if $n\geq2$, with $\varepsilon_k=+1$ for $k<n$ and $\varepsilon_n=-1$, Lemma~\ref{lem:multiquadratic-rank}(4)(c) gives $\rank E_i(L_{n,i})\geq R+2^{n-1}+2n-2$. (With all $\varepsilon_k=-1$, Lemma~\ref{lem:multiquadratic-rank}(4)(b) gives only $R+2^{n-1}$; see Remark~\ref{rem-signpatterns}.)
\end{itemize}
Choosing for each $n$ the signs that give the larger bound, we obtain fields $L_{n,i}$ of degree $2^n$ with
\[\rank E_i(L_{n,i})\geq\rank E_i(\Q)+m,\]
where $m=\max\{2n+2^{n-1}-1,2^n-1\}$ if $R$ is odd, and $m=\max\{2n,2^{n-1}+2n-2\}$ otherwise, and with $E_i(L_{n,i})_\tors\cong T$. If $\ast=\infty$, the curves $E_i/L_{n,i}$ are pairwise non-isomorphic over $\overline{\Q}$, because the $E_i$ are. This shows $\sup(\Rk^\ast_{2^n,T})\geq R+m$, and the fields $L_{n,i}$ are explicit by Theorem~\ref{maintheorem}. This completes the proof.

\subsection{Proof of Corollary~\ref{cor-highrank}.}
Let $E_{32}$ and $E_{31}'$ be as in Section~\ref{sec-examples}. We computed
$W(E_{32})=1$, and $W(E_{31}')=-1$ using Magma. Fix $n\ge1$ and
signs $\varepsilon_1,\dots,\varepsilon_n\in\{\pm1\}$. Theorem~\ref{maintheorem} gives square-free integers $d_1,\dots,d_n$
with $W(E^{d_k})=\varepsilon_k$ and $L_n=\Q(\sqrt{d_1},\dots,\sqrt{d_n})$ of degree $2^n$, whose
discriminants are prime to the conductor, and over which the torsion does not grow
(Theorem~\ref{maintheorem}(4)). The same holds for $E_{31}'$ and fields $L_n'$. By
Theorem~\ref{thm-main2}(1), we have $R_{E_{32}/L_n}\ge 32+n$ and $R_{E_{31}'/L_n'}\ge 31+n$, for
either choice of signs.

Now assume the Parity Conjecture for the twists $E^d$ with $d\in\mathcal{D}_n$, $d\ne1$.
For $E_{32}$, Theorem~\ref{thm-main2}(4) gives $R_{E_{32}/L_n}\ge 32+2n$ if all $\varepsilon_k=1$,
and, if $n\geq2$, $\ge 32+2^{n-1}+2n-2$ if $\varepsilon_k=1$ for $k<n$ and $\varepsilon_n=-1$. For $E_{31}'$, Theorem~\ref{thm-main2}(2) gives
$R_{E_{31}'/L_n'}\ge 31+2n+2^{n-1}-1$ if all $\varepsilon_k=1$, and Theorem~\ref{thm-main2}(3)
gives $\ge 31+2^n-1$ if all $\varepsilon_k=-1$. Choosing, for each $n$, the signs
that give the larger bound, we obtain $R_{E_{32}/L_n}\ge 32+\max\{2n,2^{n-1}+2n-2\}$ and
$R_{E_{31}'/L_n'}\ge 30+\max\{2n+2^{n-1},2^n\}$.

\section{Further examples and code}\label{sec-code}
The code for this paper can be found at \href{https://github.com/dylan-costa/rank-growth-multiquadratic}{github.com/dylan-costa/rank-growth-multiquadratic}. In the Magma file \texttt{EllipticRankTwists.m} we provide the following two functions.
\begin{enumerate}
\item \texttt{NaiveRootNumberTwistSearch(E, N : keep := k)}, the naive search. We write $E/\Q$ in short Weierstrass form $y^2=f(x)$, and for every integer $x$ with $|x|\leq N$ we put $d=\sqf(f(x))$. The function keeps the values $d\neq1$ relatively prime to the conductor with $W(E^d/\Q)=1$ (so that $R_{E^d/\Q}$ is even, and positive if the point coming from $x$ has infinite order), and returns the $k$ of them with the smallest $|d|$. By Remark~\ref{rem-bruteforce}, we should expect an output given a large enough value of $N$.
\item \texttt{ConstructRootNumberTwists(E, $\varepsilon$ : howmany := n)}, the factoring-free variant of the construction of Theorem~\ref{maintheorem} described in Remark~\ref{rem-factoringfree}. For $\varepsilon\in\{\pm1\}$ it returns square-free integers $d_1,\ldots,d_n$ such that $W(E^{d_k})=\varepsilon$, $E^{d_k}$ has an explicit point of infinite order, and $L_n=\Q(\sqrt{d_1},\ldots,\sqrt{d_n})$ has degree $2^n$. The argument $\varepsilon$ may also be a sequence $[\varepsilon_1,\ldots,\varepsilon_n]$ of signs, in which case $W(E^{d_k})=\varepsilon_k$ for each $k$; for instance, \texttt{ConstructRootNumberTwists(E, [1,1,-1] : howmany := 3)} gives the sign pattern of Lemma~\ref{lem:multiquadratic-rank}(4)(c) for $n=3$. Let
\[N=2\cdot3\cdot5\cdot7\prod_{p|N_E,\ p>7}p,\]
let $q_k$ be the $k$-th odd prime not dividing $N$, and let $t_k$ be the least quadratic non-residue modulo $q_k$. At step $k$ the function takes $x_1$ with
\[x_1\equiv1\bmod 8,\quad x_1\equiv1\bmod p\ (p\mid N\text{ odd}),\quad x_1\equiv1\bmod q_j\ (j<k),\quad x_1\equiv t_k\bmod q_k,\]
that is, $x_1\equiv x_{1,0}\bmod M_k$ with $M_k=8\,q_1\cdots q_k\prod_{p|N,\ p\text{ odd}}p$, where $x_{1,0}$ is the least positive solution of these congruences. It then chooses the sign of $x_1$ to be $\varepsilon_k W(E)$, takes $|x_1|$ as small as possible subject to $|x_1|\geq(|A|+|B|+1)(q_1\cdots q_kN)^2$, puts $z_k=x_1/(q_1\cdots q_kN)^2$, and returns $d_k=\sqf(f(z_k))$. For $k=1$ this is Lemma~\ref{findgoodvaluelemma} with $q=q_1$. The optional arguments \texttt{W} and \texttt{BadPrimes} supply $W(E)$ and the primes dividing $N_E$. For curves such as those in Section~\ref{sec-examples} these are known, and passing them avoids computing the conductor and the root number. In that case, $d_k$ has more than 1000 digits and its square-freeness cannot be certified, but the square class of $d_k$, and hence $E^{d_k}$ and $W(E^{d_k})$, are exact.
\end{enumerate}

\begin{example}\label{ex-x3minus2}
Let $E/\Q$ be the elliptic curve $E:y^2=x^3-2$, so $A=0$ and $B=-2$. Then $N_E=1728=2^6\cdot3^3$, $E(\Q)_\tors$ is trivial, $R_{E/\Q}=1$ and $W(E)=-1$. We find twists with $W(E^d)=1$ and $R_{E^d/\Q}\geq1$ (hence $R_{E^d/\Q}\geq2$ under the Parity Conjecture) in two ways, using the two Magma functions described above.

\medskip
\noindent\textbf{1. Naive search.} In Magma:
\begin{verbatim}
load "EllipticRankTwists.m";
E := EllipticCurve([0, 0, 0, 0, -2]);
best, all := NaiveRootNumberTwistSearch(E, 10 : keep := 10);
\end{verbatim}
The function runs over $-10\leq x\leq10$, sets $d=\sqf(f(x))$, and keeps the $d$ such that $W(E^d)=1$ and the discriminant of $\Q(\sqrt{d})$ is relatively prime to $N_E$, as required by Lemmas~\ref{kronlemma} and \ref{lem:multiquadratic-rank}. The output is sorted by $|d|$. It returns
\[d=341,-731,\]
coming from $x=7$ and $x=-9$, respectively. The value $x=3$ is skipped, since $f(3)=25$ is a square. The values $x=-6,-3,0,2,4,5,8,10$ also give $W(E^d)=1$, but the discriminant of $\Q(\sqrt{d})$ for the corresponding $d=-218,-29,-2,6,62,123,510,998$ shares a prime with $N_E=1728$, so they are not returned. For instance, $E^{-2}:y^2=x^3+16$ and $E^6:y^2=x^3-432$ have rank 0, and the points $(0,4)$ and $(12,36)$ coming from $x=0$ and $x=2$ are 3-torsion points. All other $x$ with $|x|\leq10$ give $W(E^d)=-1$. For the two values returned, the point coming from $x$ has infinite order:
\begin{center}
\begin{tabular}{|c|c|c|c|c|}
\hline
$x_0$ & $d=\sqf(f(x_0))$ & $E^d$ & point on $E^d$ & $R_{E^d/\Q}$\\
\hline
$7$ & $341$ & $y^2=x^3-79303642$ & $(2387,116281)$ & 2\\
$-9$ & $-731$ & $y^2=x^3+781235782$ & $(6579,534361)$ & 2\\
\hline
\end{tabular}
\end{center}
In both cases $W(E^d)=1$, and Magma's \texttt{RankBounds} gives $R_{E^d/\Q}=2$ exactly. Hence, by Proposition~\ref{prop-kida}, $R_{E/\Q(\sqrt{d})}=R_{E/\Q}+R_{E^d/\Q}=3$ for $d=341,-731$. These are the values found by the naive search over $|x|\leq10$. They are not the rank-2 twists of smallest $|d|$ with discriminant prime to $N_E$: for instance, $E^{53}$ also has rank 2.

\medskip
\noindent\textbf{2. The construction.} In Magma:
\begin{verbatim}
ds, recs := ConstructRootNumberTwists(E, 1 : howmany := 3);
\end{verbatim}
Here $\varepsilon=1$, and the function computes $N_E$ and $W(E)$ itself. Since the primes dividing $N_E$ are 2 and 3, we get $N=2\cdot3\cdot5\cdot7=210$, and $q_1=11$ is the smallest odd prime not dividing $N$. The smallest quadratic non-residue modulo 11 is $t_1=2$, so the conditions on $x_1$ are
\[x_1\equiv1\bmod 8,\quad x_1\equiv1\bmod 3,\quad x_1\equiv1\bmod5,\quad x_1\equiv1\bmod7,\quad x_1\equiv2\bmod11,\]
that is, $x_1\equiv2521\bmod 9240$, where $M_1=8\cdot3\cdot5\cdot7\cdot11=9240$. Let $z_1=x_1/(q_1N)^2=x_1/2310^2=x_1/5336100$. We want $W(E^{d_1})=\varepsilon=1$, and $W(E^{d_1})=\sign(d_1)W(E)$ (see below), so we need $\sign(d_1)=\varepsilon W(E)=-1$, that is, $f(z_1)<0$. The function guarantees $\sign(f(z_1))=\sign(z_1)$ by asking for $|z_1|\geq|A|+|B|+1=3$, i.e. $|x_1|\geq3\cdot5336100=16008300$. The negative element of the class $2521\bmod9240$ closest to 0 satisfying this is $x_1=2521-1733\cdot9240=-16010399$, so $z_1=-16010399/5336100$. Then $f(z_1)=F/2310^6$ with $F=x_1^3-2\cdot2310^6\equiv x_1^3$ modulo 8, 3, 5, 7 and 11, and
\[f(z_1)=-\frac{4407871453975949551199}{12326391000^2},\]
\begin{align*}
d_1&=-4407871453975949551199\\
&=-89\cdot4597\cdot10773688394668603.
\end{align*}
Hence $d_1<0$, $d_1\equiv1\bmod 8$, $\left(\frac{d_1}{3}\right)=\left(\frac{d_1}{5}\right)=\left(\frac{d_1}{7}\right)=1$ and $\left(\frac{d_1}{11}\right)=-1$. In particular $d_1$ is a fundamental discriminant prime to $N_E$, and Lemma~\ref{kronlemma} gives
\[W(E^{d_1})=W(E)\left(\frac{d_1}{-N_E}\right)=(-1)\left(\frac{d_1}{-1}\right)\left(\frac{d_1}{2}\right)^6\left(\frac{d_1}{3}\right)^3=(-1)(-1)(1)(1)=1,\]
which we also confirmed with Magma's \texttt{RootNumber}. Thus $R_{E^{d_1}/\Q}\geq1$. Assuming Parity Conjecture for elliptic curves over $\Q$, we conclude $R_{E^{d_1}/\Q}\geq2$.

For $k=2,3$ the function adds the next odd primes $q_2=13$ and $q_3=17$ not dividing $N$, asks $x_1\equiv1$ modulo 8, 3, 5, 7 and the previous $q_j$, and $x_1\equiv t_k$ modulo $q_k$ ($t_2=2$, $t_3=3$), and puts $z_k=x_1/(q_1\cdots q_kN)^2$:
\begin{center}
\begin{tabular}{|c|c|c|l|}
\hline
$k$ & $q_k$ & $x_1\bmod M_k$ & $d_k$\\
\hline
1 & 11 & $2521\bmod9240$ & $-4407871453975949551199$\\
 & & & $\ =-89\cdot4597\cdot10773688394668603$\\
\hline
2 & 13 & $36961\bmod120120$ & $-21268670741331361021169357399$\\
 & & & $\ =-23248001011\cdot914860195131095309$\\
\hline
3 & 17 & $1921921\bmod2042040$ & $-513363857667648563737041085212527519$\\
 & & & $\ =-642907\cdot47621000695841\cdot16767897791344237$\\
\hline
\end{tabular}
\end{center}
Each $d_k$ is negative, $d_k\equiv1\bmod8$ and $\left(\frac{d_k}{3}\right)=\left(\frac{d_k}{5}\right)=\left(\frac{d_k}{7}\right)=1$, so $W(E^{d_k})=1$, which we also confirmed directly with Magma. Further, $\left(\frac{d_k}{q_j}\right)=1$ for $j<k$ and $\left(\frac{d_k}{q_k}\right)=-1$. In particular, $\left(\frac{d_2}{11}\right)=1$, $\left(\frac{d_2}{13}\right)=-1$, and $\left(\frac{d_3}{11}\right)=\left(\frac{d_3}{13}\right)=1$, $\left(\frac{d_3}{17}\right)=-1$. If a non-empty product $\prod_{k\in S}d_k$ were a square, its Kronecker symbol at $q_m$, $m=\min S$, would be $-1$. So no such product is a square, and $L_3=\Q(\sqrt{d_1},\sqrt{d_2},\sqrt{d_3})$ has degree 8 over $\Q$. Since every $d_k$ is $\equiv1\bmod8$ and prime to $210N_E$, every quadratic subfield of $L_3$ is ramified at a prime not dividing $210N_E$, so $E(L_3)_\tors=E(\Q)_\tors$ is trivial by Corollary~\ref{cor-fujita} (see Remark~\ref{rem-factoringfree}). Finally, $W(E)=-1$, $W(E^{d_k})=1$, and the discriminant of $\Q(\sqrt{d_k})$ is $d_k$, which is prime to $N_E$. So Lemma~\ref{lem:multiquadratic-rank} gives $R_{E/L_3}\geq1+3=4$, and
\[R_{E/L_3}\geq R_{E/\Q}+2\cdot3+2^2-1=10\]
if the Parity Conjecture holds for the twists $E^d$, $d\in\mathcal{D}_3$, $d\neq1$.

If instead we ask for root number $-1$, the command
\begin{verbatim}
ds, recs := ConstructRootNumberTwists(E, -1 : howmany := 3);
\end{verbatim}
uses the same classes modulo $M_k$ but takes $x_1>0$, and returns the positive values
\begin{align*}
d_1&=157\cdot3271\cdot7407287562553843\\
d_2&=263\cdot4177\cdot16691764119808305483311\\
d_3&=29\cdot43\cdot354895667836067582927094829938383,
\end{align*}
each with $W(E^{d_k})=-1$. Now every $d\in\mathcal{D}_3$ has $W(E^d)=-1$, and part (3) of Lemma~\ref{lem:multiquadratic-rank} gives $R_{E/L_3}\geq1+2^3-1=8$ under the Parity Conjecture, which for $n=3$ is weaker than the bound 10 above.
\end{example}

We have included a Magma file \texttt{HighRankCurves.m} that contains examples of elliptic curves in Section~\ref{sec-examples}. The following table gives, for each of them, the congruence class of $x_{1,0}$ modulo $M_1=8\,q_1\prod_{p\mid N,\ p\text{ odd}}p$ used in the first step of \texttt{ConstructRootNumberTwists(E, 1)}, where $N=2\cdot3\cdot5\cdot7\prod_{p\mid N_E,\ p>7}p$, $q_1$ is the smallest odd prime not dividing $N$, and $t_1$ is the least quadratic non-residue modulo $q_1$. The modulus $M_1$ itself has between $122$ and $181$ digits, so we do not print it. We checked with Magma that each printed class is $\equiv1$ modulo $8$ and modulo every odd prime dividing $N$, is $\equiv t_1\bmod q_1$, and is smaller than $M_1$.

{\footnotesize
\setlength{\tabcolsep}{4pt}
\begin{longtable}{|l|c|c|c|p{0.62\textwidth}|}
\hline
Curve & Rank $E$ & $q_1$ & $t_1$ & Congruence class of $x_{1,0}\bmod M_1$\\
\hline
\endfirsthead
\hline
Curve & Rank $E$ & $q_1$ & $t_1$ & Congruence class of $x_{1,0}\bmod M_1$ (continued)\\
\hline
\endhead
\hline
\endfoot
\endlastfoot
\#4239 & 32 & 19 & 2 & \texttt{417936778031469211008423692067561218391040323671089446492773\newline 464417872540726825409924423817705139806339508710699447616511\newline 928110993671212051702673195059158353808066765235730699715248\newline 1}\\
\hline
\#3900 & 31 & 23 & 5 & \texttt{151563742734514126176461046353466154640950934466514102903425\newline 769898631318294280915589716900493317201380225846060033973718\newline 282017093504477014054121}\\
\hline
\#302 & 31 & 17 & 3 & \texttt{434785926665697132617222008787064545487384851175023935487161\newline 782959380329205914150078248210725112231590402712384999809986\newline 85490105258461299982095106982893653299817841}\\
\hline
\#398 & 30 & 11 & 2 & \texttt{680555440448938192924961023960728617033424863956231038163621\newline 590338524928590013360686392469336637849329668092417596248552\newline 54634004390641}\\
\hline
\#728 & 29 & 11 & 2 & \texttt{104527125721949699947599772717711850658644801899358488971116\newline 493190021795026452399538456828540694381511209456477728097609\newline 21}\\
\hline
\#11 & 28 & 23 & 5 & \texttt{101364299178488055351612780037113753448489670836616172074417\newline 458891182911926922727971151188396214862793375684996910938198\newline 985885659464046798820721}\\
\hline
\end{longtable}
}

\section{Finding twists of rank greater than $2$}\label{sec-rank3}

Theorem~\ref{thm-main2} produces twists $E^{d_k}$ with one explicit point each. When $W(E^{d_k}/\Q)=1$, the Parity Conjecture then gives $R_{E^{d_k}/\Q} \geq 2$. In this section we explain how, in some cases, one can obtain twists of rank $\geq 2$ unconditionally (with two independent rational points that are explicitly given), and conjecturally of rank $\geq 3$, using a classical construction of Mestre. However, we also explain that this method cannot be extended to all elliptic curves over $\Q$, while our methods above apply to all of them. Mestre's construction was successfully used in the authors' preprint \cite{CostaLozanoRobledo-rank35} to show that there exists an explicit quadratic field $K$ such that the curve $E_{32}$ of Theorem~\ref{thm-E32} has rank $\geq 34$ unconditionally and $\geq 35$ assuming Parity or finiteness of a $p$-primary component of $\Sha$ over the quadratic field $K$.

\subsection{Mestre's family}
Let $E/\Q$ be given by $y^2=x^3+ax+b$ with $a,b\in\Z$. Define
\[g(t)=-ab\,(t^2+1)\,I(t),\qquad I(t)=b^2(t^4+t^2+1)^3+a^3t^4(t^2+1)^2,\]
a polynomial of degree $14$ with $g(-t)=g(t)$ and $g(1/t)=t^{-14}g(t)$.

\begin{theorem}[Mestre \cite{mestre}; see also \cite{rubinsilverberg,stewarttop,gusictadic}]\label{thm-mestre}
Suppose $ab\neq 0$. Then, the quadratic twist $E^{g(t)}\colon y^2=x^3+ag(t)^2x+bg(t)^3$ has rank at least $2$ over $\Q(t)$. Two independent points are
\begin{align*}
P(t)&=\left(-\frac ba\cdot\frac{t^4+t^2+1}{t^2+1}\, g(t),\ \frac{g(t)^2}{a^2(t^2+1)^2}\right),\\
Q(t)&=\left(-\frac ba\cdot\frac{t^4+t^2+1}{t^2(t^2+1)}\, g(t),\ \frac{g(t)^2}{a^2t^3(t^2+1)^2}\right).
\end{align*}
\end{theorem}

By Silverman's specialization theorem \cite{silvermanspecialization}, $P(t_0)$ and $Q(t_0)$ are independent in $E^{g(t_0)}(\Q)$ for all but finitely many $t_0\in\Q$.

\begin{corollary}\label{cor-mestre-rank3}
Let $E/\Q$ be an elliptic curve that satisfies the conditions of Theorem \ref{thm-mestre}. Let $t_0\in\Q$ with $g(t_0)\neq0$ and $P(t_0)$, $Q(t_0)$ independent, and let $K=\Q(\sqrt{g(t_0)})$. Then $R_{E/K}\geq R_{E/\Q}+2$. If $W(E^{g(t_0)})=-1$, then the $p^\infty$-Selmer corank of $E^{g(t_0)}/\Q$ is odd and $\geq 3$ for every prime $p$. Therefore, assuming Conjecture \ref{conj-parity}, or $\Sha(E^{g(t_0)}/\Q)[p^\infty]$ is finite for some $p$, then $R_{E^{g(t_0)}/\Q}\geq 3$ and $R_{E/K}\geq R_{E/\Q}+3$.
\end{corollary}

Ideally, we would like to use Corollary \ref{cor-mestre-rank3} to construct a sequence of square-free integers $d_k$ to be used as in Lemma \ref{lem:multiquadratic-rank}, but several difficulties arise. For instance, Corollary \ref{cor-mestre-rank3} does not work for every elliptic curve over $\Q.$

\begin{example}[$j=0$ and $j=1728$]\label{ex-mestre-j0}
If $a=0$ (that is, $j=0$) or $b=0$ ($j=1728$), then $g(t)\equiv 0$ and Theorem~\ref{thm-mestre} does not apply. This is the case for the curve $E\colon y^2=x^3-2$ of Examples~\ref{ex-x3minus2-intro} and~\ref{ex-x3minus2}, which has $j=0$; similarly for any curve $y^2=x^3+Ax$, which has $j=1728$. For these curves, Theorem~\ref{thm-main2} still applies and produces explicit twists of positive rank and prescribed root number.
\end{example}

Another possible situation is when the Mestre twist always gives $W(E^{g(t)}/\Q)=1$ for every $t\in\Q$ with $g(t)\neq 0$. For instance, one can show this to be the case for the curve \texttt{162c1}. Then we cannot guarantee even under Conjecture \ref{parity} that any of the quadratic twists produced by the Mestre construction have rank $\geq 3$. As a future direction, the authors are currently looking into sufficient conditions on an elliptic curve $E/\Q$ to apply Mestre's construction and conclude certain conditional rank increases in multi-quadratic extensions by multiples of, conjecturally, at least $3$.


\bibliography{bibliography}
\bibliographystyle{plain}


\end{document}